\documentclass[11pt, reqno]{amsart}

\usepackage{amsmath,amssymb,amsthm, epsfig}
\usepackage{hyperref}
\usepackage{amsmath}
\usepackage{amssymb}
\usepackage{color}
\usepackage{ulem}
\usepackage{epsfig}
\usepackage[mathscr]{eucal}

\def \dist {\mathrm{dist}}

\newtheorem{theorem}{Theorem}[section]
\newtheorem{lemma}[theorem]{Lemma}

\newtheorem{corollary}[theorem]{Corollary}
\theoremstyle{definition}

\newtheorem{definition}[theorem]{Definition}

\theoremstyle{remark}
\newtheorem{remark}[theorem]{Remark}

\numberwithin{equation}{section}

\newcommand{\intav}[1]{\mathchoice {\mathop{\vrule width 6pt height 3 pt depth  -2.5pt
\kern -8pt \intop}\nolimits_{\kern -6pt#1}} {\mathop{\vrule width
5pt height 3  pt depth -2.6pt \kern -6pt \intop}\nolimits_{#1}}
{\mathop{\vrule width 5pt height 3 pt depth -2.6pt \kern -6pt
\intop}\nolimits_{#1}} {\mathop{\vrule width 5pt height 3 pt depth
-2.6pt \kern -6pt \intop}\nolimits_{#1}}}

\begin{document}

\title[Anisotropic fully nonlinear equations]{Regularity for anisotropic fully nonlinear\\
integro-differential equations}

\author[L. Caffarelli]{Luis A. Caffarelli}
\address{Department of Mathematics, University of Texas at Austin, 1 University Station, C1200, Austin, Texas 78712, USA.}
\email{caffarel@math.utexas.edu}

\author[R. Leit\~ao]{Raimundo Leit\~ao}
\address{CMUC, Department of Mathematics, University of Coimbra, 3001-501 Coimbra, Portugal.}
\email{raimundo@mat.uc.pt}

\author[J.M. Urbano]{Jos\'e Miguel Urbano}
\address{CMUC, Department of Mathematics, University of Coimbra, 3001-501 Coimbra, Portugal.}
\email{jmurb@mat.uc.pt}

\begin{abstract}

We consider fully nonlinear integro-differential equations governed by kernels that have different homogeneities in different directions. We prove a nonlocal version of the ABP estimate, a Harnack inequality and the interior $C^{1, \gamma}$ regularity, extending the results of \cite{CS} to the anisotropic case. 

\bigskip


\noindent{\sc AMS Subject Classification MSC 2010}: 35J60, 47G20, 35D40, 35B65

\end{abstract}

\maketitle

\section{Introduction}

In this work we develop a regularity theory for elliptic fully nonlinear integro-differential equations of the type
\begin{equation}\label{def F1}
Iu\left( x \right) := \inf\limits_{\alpha}\sup\limits_{\beta} L_{\alpha \beta}u\left( x \right)= 0, 
\end{equation}
where 
\begin{equation}\label{operator}
L_{\alpha \beta}u\left( x \right) :=  \int_{\mathbb{R}^{n}}\left( u\left( x + y \right) - u\left(x \right) - \nabla u \left( x \right)\cdot y \chi_{B_{1}}\left( y \right)\right) K_{\alpha\beta}\left( y\right)dy,
\end{equation}
and the kernels $K_{\alpha\beta}$ are symmetric and satisfy the anisotropic bounds
\begin{equation}\label{Kernel cond 2}  
\dfrac{\lambda c_{\sigma} }{\sum_{i=1}^{n}|y_{i}|^{n+\sigma_{i}}} \leq K_{\alpha\beta}\left( y \right) \leq   \dfrac{\Lambda c_{\sigma}}{\sum_{i=1}^{n}|y_{i}|^{n+\sigma_{i}}}, \quad \forall y \in \mathbb{R}^{n},
\end{equation}
for $0< \lambda \leq \Lambda $, $0 < \sigma_{i} < 2$, and $c_{\sigma} = c\left( \sigma_{1}, \dots, \sigma_{n} \right) > 0$ a normalization constant.

Equations of type \eqref{def F1} appear extensively in the context of stochastic control problems (see \cite{SO}), namely in competitive stochastic games with two or more players, which are allowed to choose from different strategies at every step in order to maximize the expected value of some function at the first exit point of a domain. Integral operators like \eqref{operator} correspond to purely jump processes when diffusion and drift are neglected. The anisotropic setting we consider is bound to be of use in the context of financial mathematics, namely for Black-Scholes models that use certain jump-type processes instead of diffusions (cf. \cite{R}).

The isotropic version of the problem, with \eqref{Kernel cond 2} replaced by
\begin{equation}
 \label{isotropic kernel}  \dfrac{\lambda \left( 2- \sigma \right)}{|y|^{n+\sigma}} \leq K_{\alpha\beta}\left( y \right) \leq   \dfrac{\Lambda  \left( 2- \sigma \right)}{|y|^{n+\sigma}}, \quad \forall y \in \mathbb{R}^{n},
\end{equation}
for $0< \sigma < 2$, is studied in \cite{CS}, exploring the analogy between ellipticity and the condition 
$$M^{-}_{\mathcal{L}}v\left( x \right) \leq I\left( u + v \right)\left( x \right)  - Iv\left( x \right)  \leq   M^{+}_{\mathcal{L}}v\left( x \right), \quad \ \forall y \in \mathbb{R}^{n}.$$
Here, $\mathcal{L}$ is the class of operators $L_{\alpha\beta}$ whose kernels satisfy \eqref{isotropic kernel} and the operators 
$$M^{-}_{\mathcal{L}} u \left( x \right) := \inf \limits_{L \in \mathcal{L}} Lu\left( x \right) \quad \text{and} \quad M^{+}_{\mathcal{L}} u \left( x \right) := \sup\limits_{L \in \mathcal{L}} Lu\left( x \right)$$    
correspond to the extremal Pucci operators in the theory of elliptic equations of second order. The non-variational approach to regularity theory for (sub and super) viscosity solutions of the isotropic version of equation \eqref{def F1} is a nonlocal version of the strategy used in \cite{CC} for second order fully nonlinear elliptic equations. 

In the classical non-variational approach, the crucial step towards a regularity theory is the celebrated Aleksandrov--Bakel'man--Pucci estimate (ABP estimate, in short), which amounts to the bound
\begin{equation} \label{abpzinha}
\sup_{B_1} u \leq  C\left( n \right) \left( \int_{\left\lbrace \Gamma = u \right\rbrace \cap B_{1}} \left( f^{+}\right)^n \right)^{1/n},
\end{equation}
for any viscosity subsolution $u$ of the maximal Pucci equation with right-hand side $(-f)$ taking non-positive values outside the unit ball $B_{1}$ . Here, $\Gamma$ is the concave envelope of $u$ in $B_{3}$.  The technical advantage of the ABP estimate stems from relating a pointwise estimate with an estimate in measure. More precisely, $u\left( 0 \right) \geq 1$ implies  
$$1 \leq  C\Vert f \Vert_{L^{\infty}} \left |  \left\lbrace \Gamma = u \right\rbrace \cap B_{1} \right |^{\frac{1}{n}} \leq C \Vert f \Vert_{L^{\infty}} \left |  \left\lbrace u \geq 0 \right\rbrace \cap B_{1} \right |^{\frac{1}{n}}.  $$

In the nonlocal setting, the ABP estimate must be modified in face of the structural differences of the operator. In the isotropic case of \cite{CS}, we have to replace \eqref{abpzinha} by the following two assertions, which still give access to the regularity theory: 

\begin{enumerate}
\item [i.] $u$ stays quadratically close to the tangent plane to $\Gamma$ in a large portion of a neighbourhood around a contact point:
$$ \left | \left\lbrace y \in 8\sqrt{n}Q_{j} : u\left( y \right) \geq \Gamma \left( y \right) - \left( \max_{\overline{Q}_{j}} f^{+} \right)d^{2}_{j} \right\rbrace \right | \geq \varsigma \left | Q_{j} \right |;$$

\item [ii.] $\Gamma$ has quadratic growth and therefore
$$ \left | \nabla \Gamma \left( \overline{Q}_{j} \right) \right | \leq C \left( \max_{\overline{Q}_{j}} f^{+} \right)^{n} \left | \overline{Q}_{j} \right |,$$
\end{enumerate}
for a finite family of disjoint open cubes $\left\lbrace Q_{j}\right\rbrace$ with diameters $d_{j} \leq  \frac{1}{8\sqrt{n}}$ such that 
$$
\left\lbrace  u = \Gamma \right\rbrace \subset \bigcup_{j} \overline{Q}_{j} \quad \text{and} \quad \left\lbrace  u = \Gamma \right\rbrace \cap \overline{Q}_{j} \neq \emptyset,
$$
where $\nabla \Gamma$ stands for any element of the superdifferential of $\Gamma$, and the constants $\varsigma >0$ and $C>0$ only depend on dimension and the ellipticity constants.

Then, using i. and ii., we get from $u(0) \geq 1$,
\begin{eqnarray} \label{Use Non ABP}
1 & \leq & C\Vert f \Vert_{L^{\infty}} \left |  \left\lbrace u \geq \Gamma - \frac{1}{64n} \Vert f \Vert_{L^{\infty}}\right\rbrace \cap B_{1}\right |^{\frac{1}{n}} \nonumber \\
& \leq & C\Vert f \Vert_{L^{\infty}} \left |  \left\lbrace u \geq - \frac{1}{64n}\Vert f \Vert_{L^{\infty}}\right\rbrace \cap B_{1}\right |^{\frac{1}{n}},  
\end{eqnarray}
which is still enough to complete a regularity theory. A covering lemma by open cubes $Q_{j}$ that satisfy assertions i. and ii. is crucial in obtaining \eqref{Use Non ABP} in the nonlocal case, for which the classical inequality \eqref{abpzinha} does not hold.

To treat the anisotropic case we use the same strategy as in \cite{CS} but the anisotropic geometry driven by the kernels $K_{\alpha \beta}$ requires a refinement of the techniques. We comment in the sequel on the main difficulties we came across and how to overcome them.

\begin{enumerate}

\medskip

\item \textit{Assertion i.} At this step of the analysis, the challenge is to find the suitable geometry of the neighbourhoods of the contact points within which there is a (large) portion where a subsolution $u$ stays quadratically close to the tangent plane to $\Gamma$ and such that, in smaller neighbourhoods (with the same geometry), the concave envelope $\Gamma$ has quadratic growth. A careful analysis of the anisotropic nonlocal version of inequality $M^{+}_{\mathcal{L}} u \geq - f$ satisfied by $u$ at the contact points allows us to conclude that the appropriate geometry is the geometry determined by the level sets of the kernels $K_{\alpha \beta}$:
$$
 \Theta_{r}\left( x \right) := \left\lbrace \left( y_{1}, \dots, y_{n}\right) \in \mathbb{R}^{n} : \sum_{i=1}^{n} | y_{i} - x_{i} |^{n+\sigma_{i}} < r \right\rbrace, 
$$
for $x \in \left\lbrace \Gamma = u \right\rbrace \cap B_{1}$.
It is also here that we choose the appropriate normalisation constant:
$$
c_{\sigma} = -1 + \frac{3}{n+\sigma_{\max}} + \sum\limits_{\sigma_{j}\neq \sigma_{\max}}\frac{1}{n+\sigma_{j}}.
$$

\medskip

\item \textit{Assertion ii.} Given a positive number $h>0$, a fine analysis allows us to conclude that if a concave function, for instance the concave envelope $\Gamma$, remains below its tangent plane translated by $-h$ in a (universally sufficiently small in measure) portion of a (sufficiently large) annulus of the unit ball, for example $B_{1}\setminus B_{\frac{1}{2}}$, then $\Gamma + h$ is above its tangent plane in the interior ball of the annulus, in this case $B_{\frac{1}{2}}$. In the anisotropic case, the difficulty is to extend this argument to the anisotropic balls $\Theta_{r}$. Through the anisotropic transformation $T:\mathbb{R}^{n} \rightarrow \mathbb{R}^{n}$, defined by
$$
Te_{i}:= r^{\frac{1}{n+\sigma_{i}}}e_{i},
$$
and taking into account that the composition of a concave function with an affine function is still concave, we extend this fine analysis to ellipses. We then use the previous step and the symmetry of the anisotropic balls $\Theta_{r}$ with respect to $x$ to conclude that $\Gamma$ grows quadratically in such anisotropic balls. 

\medskip

\item \textit{Covering Lemma.} In \cite{CS}, the Besicovitch Covering Lemma is used. Our covering is naturally made of $n$-dimensional rectangles $\mathcal{R}_{j}$ and we invoke a covering lemma from \cite{CCal}. We stress that this covering lemma allows for a change of direction in the homogeneity degrees $\sigma_{i}$, but each $\sigma_{i}$ must remain constant. Degenerate spatial changes of the homogeneities $\sigma_{i}$, arising for example in the context of spherical operators or other special weights, would require the use of a more general covering lemma like the one in \cite{CCal2}. In adapting our results to that case, the main difficulty lies in the use of the barriers and we plan to address this issue in a forthcoming paper.
 
\end{enumerate}

\medskip

\noindent With this at hand, we then use the natural anisotropic scaling to build an adequate barrier function and, together with the nonlocal anisotropic version  of the ABP estimate, we prove a lemma that links a pointwise estimate with an estimate in measure, Lemma \ref{Point Estimates 1}. This is the fundamental step towards a regularity theory. The iteration of Lemma \ref{Point Estimates 1} implies the decay of the distribution function $\lambda_{u}:= \left| \left\lbrace u > t \right\rbrace \right |$ and the tool that makes this iteration possible is the so called Calder\'on -Zygmund decomposition. Since our scaling is anisotropic we need a Calder\'on -Zygmund decomposition for $n$-dimensional rectangles generated by our scaling. A fundamental device we use for that decomposition is the Lebesgue differentiation theorem for $n$-dimensional rectangles that satisfy the condition of Caffarelli-Calder\'on in \cite{CCal}. Then we prove the Harnack inequality and, as a consequence, we obtain the interior $C^{\gamma}$ regularity for a solution $u$ of equation \eqref{def F1} and, under additional assumptions on the kernels $K_{\alpha \beta}$, interior $C^{1, \gamma}$ estimates.

We finally observe that the power of the estimates obtained in \cite{CS} is revealed as $\sigma \rightarrow 2$. In fact, since the estimates remain uniform in the degree $\sigma$, it was possible to obtain an interesting relation between the theory of integro-differential equations and that of elliptic differential equations through the natural limit:
$$
\lim \limits_{\sigma \rightarrow 2} \int_{\mathbb{R}^{n}}\dfrac{c_{n}\left( 2 - \sigma \right)}{\vert y \vert^{n+\sigma}} \left( u \left( x + y \right) + u \left( x - y \right) - 2u \left( x \right) \right) dy$$
$$ = \lim \limits_{\sigma \rightarrow 2} - \left( - \Delta \right)^{\frac{\sigma}{2}}u\left( x \right) = \Delta u\left( x \right),   
$$
where $c_{n} > 0$ is a constant. This contrasts with previous results in the literature on Harnack inequalities and H\"older estimates for integro-differential equations, with either analytical proofs \cite{S} or
probabilistic proofs \cite{BK, BK2, BL, Son}, whose estimates blow up as the order of the equation approaches $2$. We emphasize that our estimates are also stable as $\sigma_{\min}:=\min\left\lbrace \sigma_{1}, \dots, \sigma_{n} \right\rbrace$ approaches 2. 

The paper is organised as follows. In section \ref{preliminaries} we gather all the necessary tools for our analysis: the notion of viscosity solution for the problem \eqref{def F1}, the extremal operators of Pucci type associated with the family of kernels $K_{\alpha \beta}$ and some notation. Section \ref{ABP Estimate}, where the nonlocal ABP estimate for a solution $u$ of equation \eqref{def F1} is obtained, is the most important of the paper. Sections \ref{Barrier function section} and \ref{Harnack Inequality Section} are devoted to the proof of the Harnack inequality and its consequences. 
 
\section{Viscosity solutions and extremal operators}\label{preliminaries}

In this section we collect the technical properties of the operator $I$ that we will use throughout the paper. Since $K_{\alpha\beta}$ is symmetric and positive, we have
$$
L_{\alpha \beta}u\left( x \right) =  PV \int_{\mathbb{R}^{n}}\left( u\left( x + y \right) - u\left(x \right) \right) K_{\alpha\beta}\left( y\right)dy
$$
and
$$
L_{\alpha \beta}u\left( x \right) =  \dfrac{1}{2} \int_{\mathbb{R}^{n}}\left( u\left( x + y \right) - u\left(x - y \right) - 2\left( x \right) \right) K_{\alpha\beta}\left( y\right)dy.
$$
For convenience of notation, we denote
$$
\delta \left( u, x, y \right):= u \left( x + y \right) + u \left( x - y \right) - 2u \left( x \right)
$$
and we can write
$$
L_{\alpha \beta} = \int_{\mathbb{R}^{n}}\delta \left( u, x, y \right) K_{\alpha\beta}\left( y\right)dy,
$$
for some kernel $K_{\alpha\beta}$.

We now define the adequate class of test functions for our operators.
 
\begin{definition}
\label{C^{1,1} definition}
A function $\phi$ is said to be $C^{1,1}$ at the point $x$, and we write $\phi \in C^{1,1}\left( x \right)$, if there is a vector $v \in \mathbb{R}^{n}$ and numbers $M, \eta_{0} >0$ such that 
$$
\vert \phi \left( x + y\right) - \phi \left( x \right) -  v \cdot y   \vert \leq M \vert y \vert^{2},
$$
for $\vert  x \vert < \eta_{0}$. We say that a function $\phi$ is $C^{1,1}$ in a set $\Omega$, and we denote $\phi \in C^{1,1}\left( \Omega \right)$, if the previous holds at every
point, with a uniform constant $M$.
\end{definition}

\begin{remark}
Let $u \in C^{1,1}\left( x \right) \cap L^{\infty}\left( \mathbb{R}^{n} \right)$ and $M>0$ and $\eta_{0}>0$ be as in definition \ref{C^{1,1} definition}. Then we estimate 
$$L_{\alpha \beta}u\left( x \right)  =  PV \int_{\mathbb{R}^{n}}\delta \left( u, x, y \right) K_{\alpha\beta}\left( y\right)dy$$
$$\leq  \left[  4c_{\sigma}\Lambda|u|_{L^{\infty}\left( \mathbb{R}^{n} \right)}\int_{\mathbb{R}^{n} \setminus B_{\eta_{0}} } \dfrac{1}{\sum_{i=1}^{n}|y_{i}|^{n+\sigma_{i}}} dy + 2Mc_{\sigma}\Lambda\int_{B_{\eta_{0}}} \dfrac{ |y|^{2}}{\sum_{i=1}^{n}|y_{i}|^{n+\sigma_{i}}} dy \right] $$
$$ \leq    \left[ 4c_{\sigma}\Lambda|u|_{L^{\infty}} 2^{\frac{n+2}{2}} \eta_{0}^{-\left( \sigma_{\max} - \sigma_{\min}\right)} \int_{\mathbb{R}^{n} \setminus B_{\eta_{0}}} \dfrac{1}{|y|^{n+\sigma_{\min}}} dy +  C\left(  n, \Lambda, M, \eta_{0} \right) \right] $$ 
$$ = \left[c_{\sigma}C\left( n, \Lambda, |u|_{L^{\infty}} \right) \frac{\eta_{0}^{- \sigma_{\max}}}{\sigma_{\min}} + C\left(  n, \Lambda, M, \eta_{0} \right)\right] $$
and conclude that $Iu\left( x \right) \in \mathbb{R}$. 
\end{remark}

We now introduce the notion of viscosity subsolution (and supersolution) $u$ in a domain $\Omega$, with $C^{2}$ test functions that touch $u$ from above or from below. We stress that $u$ is allowed to have arbitrary discontinuities outside of $\Omega$.

\begin{definition}
Let $f$ be a bounded and continuous function in $\mathbb{R}^{n}$. A function $u:\mathbb{R}^{n} \rightarrow \mathbb{R}$, upper (lower) semicontinuous in $\overline{\Omega}$, is said to be a subsolution (supersolution) to equation $Iu = f$, and we write $Iu \geq f$ ($Iu \leq f$), if whenever the following happen:  
\begin{enumerate}
\item $x_{0} \in \Omega$ is any point in $\Omega$;
\medskip
\item $B_{r}\left( x_{0} \right) \subset \Omega$, for some $r>0$;
\medskip
\item $\phi \in C^{2}\left(\overline{B_{r}\left( x_{0} \right)} \right)$;
\medskip
\item $\phi\left( x_{0} \right) = u\left( x_{0} \right)$;
\medskip
\item $\phi \left( y \right) > u\left( y \right)$ ($\phi \left( y \right) < u\left( y \right)$) for every $y \in B_{r}\left( x_{0} \right) \setminus \left\lbrace x_{0} \right\rbrace$; 
\end{enumerate}
then, if we let 
$$
v := \left \{ 
\begin{array}{lll}
\phi , & \text{ in } & B_{r}\left( x_{0} \right) \\
u & \text{ in } & \mathbb{R}^{n} \setminus  B_{r}\left( x_{0} \right), 
\end{array}
\right.
$$
we have $Iv\left( x_{0} \right) \geq f\left( x_{0} \right)$ ($Iv\left( x_{0} \right) \leq f\left( x_{0} \right)$).
\end{definition}

\begin{remark}
 Functions which are $C^{1,1}$ at a contact point $x$ can be used as test functions in the definition of viscosity solution (see Lemma 4.3 in \cite{CS}).
\end{remark}
Next, we define the class of linear integro-differential operators that will be a fundamental tool for the regularity analysis. Let $\mathcal{L}_{0}$ be the collection of linear operators $L_{\alpha \beta}$. We define the maximal and minimal operator with respect to $\mathcal{L}_{0}$ as
$$
M^{+} u \left( x \right) := \sup\limits_{L \in \mathcal{L}_{0}} Lu\left( x \right) 
$$
and
$$
M^{-} u \left( x \right) := \inf \limits_{L \in \mathcal{L}_{0}} Lu\left( x \right).
$$

By definition, if $M^{+} u \left( x \right) < \infty$ and $M^{-} u \left( x \right) < \infty$, we have the simple form
$$
M^{+} u \left( x \right) = c_{\sigma} \int_{\mathbb{R}^{n}} \dfrac{\Lambda \delta^{+} - \lambda \delta^{-}}{\sum_{i=1}^{n}|y_{i}|^{n+\sigma_{i}}} dy
$$
and
$$
M^{-} u \left( x \right) = c_{\sigma}  \int_{\mathbb{R}^{n}} \dfrac{\lambda \delta^{+} - \Lambda \delta^{-}}{\sum_{i=1}^{n}|y_{i}|^{n+\sigma_{i}}} dy.
$$

\begin{remark}
As in \cite{CS}, we could consider equation \eqref{def F1} for a  more general class $\mathcal{L}$ satisfying
$$
\int_{\mathbb{R}^{n}}\dfrac{\vert y \vert^{2}}{1 + \vert y \vert^{2}}K\left( y \right)dy < \infty,
$$
where $K\left( y \right):= \sup \limits_{\alpha \in \mathcal{L}} K_{\alpha}\left( y \right)$ and $K_{\alpha}\left( y \right) = K_{\alpha}\left( - y \right)$. 
\end{remark}

The proofs of the results that we now present can be found in the sections $3$, $4$ and $5$ of \cite{CS}. The first result ensures that if $u$ can be touched from above, at a point $x$, with a paraboloid then $Iu\left( x\right)$ can be evaluated classically.

\begin{lemma}
\label{clas sense max}
If we have a subsolution, $Iu \geq f$ in $\Omega$, and $\phi$ is a $C^{2}$ function that touches $u$ from above at a point $x \in \Omega$, then $Iu\left( x \right)$ is defined in the classical sense and $Iu\left( x\right) \geq f\left( x \right)$.
\end{lemma}

Another important property of $I$ is the continuity of $I\phi$ in $\Omega$ if $\phi \in C^{1, 1}\left( \Omega \right)$.

\begin{lemma}
\label{I is C^{1,1}}
Let $v$ be a
bounded function in $\mathbb{R}^{n}$ and $C^{1,1}$ in some open set $\Omega$. Then $Iv$ is continuous in $\Omega$.
\end{lemma}

The next lemma allows us to conclude that the difference between a subsolution of the maximal operator $M^{+}$ and a supersolution of the minimal operator $M^{-}$ is a subsolution of the maximal operator.

\begin{lemma}
\label{Comp. princ.}
Let $\Omega$ be a bounded open set and $u$ and $v$ be two bounded functions in $\mathbb{R}^{n}$ such that

\begin{enumerate}

\item $u$ is upper-semicontinuous and $v$ is lower-semicontinuous in $\overline{\Omega}$;
\medskip
\item $Iu \geq f$ and $Iv \leq g$ in the viscosity sense in $\Omega$ for two continuous functions $f$ and $g$.

\end{enumerate}
Then 
$$M^{+}\left( u - v \right) \geq f - g \quad \mathrm{in} \ \ \Omega$$
in the viscosity sense.

\end{lemma}

We conclude this section introducing some notation that will be instrumental in the sequel.
Given $ r, s > 0$ and $x \in \mathbb{R}^{n}$, we will denote 
$$
E^{}_{r, s}\left( x \right) := \left\lbrace \left( y_{1}, \dots, y_{n}\right) \in \mathbb{R}^{n} : \sum_{i=1}^{n} \frac{\left( y_{i} - x_{i} \right)^{2}}{r^{\frac{2}{n+\sigma_{i}}}} < s^{2} \right\rbrace
$$
and
$$
R^{}_{r,s}\left( x \right) := \left\lbrace \left( y_{1}, \dots, y_{n}\right) \in \mathbb{R}^{n} : | y_{i} - x_{i} | < s^{\frac{1}{n+\sigma_{\min}}} r^{\frac{1}{n+\sigma_{i}}} \right\rbrace. 
$$
Given the box $R^{}_{r,s}$, we define the corresponding box  $\tilde{R}^{}_{r,s}$ by
$$
\tilde{R}^{}_{r,s}\left( x \right) := \left\lbrace \left( y_{1}, \dots, y_{n}\right) \in \mathbb{R}^{n} : | y_{i} - x_{i} | <  \left( s r \right) ^{\frac{1}{n+\sigma_{i}}} \right\rbrace. 
$$
If $\sigma_{\min}:=\min\left\lbrace \sigma_{1}, \dots, \sigma_{n} \right\rbrace$ we define
$$
i_{\min}:= \min\left\lbrace j :  \sigma_{\min}= \sigma_{j}  \right\rbrace.
$$

\begin{remark}
Let $r>0$. Hereafter, we will use the following relations:
\begin{enumerate}
\item $E^{}_{r,\frac{1}{2}} \subset \Theta_{r} \subset E^{}_{r,\sqrt{n}}$;
\medskip
\item $\Theta_{2^{-\mathfrak{C}}r} \subset E^{}_{r,\frac{1}{8}}$, for some natural number $\mathfrak{C}=\mathfrak{C}\left( n \right) > 0$;
\medskip
\item $R^{}_{r,s} \subset \tilde{R}^{}_{r,s}$, if $0 < s < 1$. 
\end{enumerate}
\end{remark}

\section{Nonlocal anisotropic ABP estimate} \label{ABP Estimate}
 
Let $u$ be a non positive function outside the ball $B_{1}$. We define the concave envelope of $u$ by 
$$
\Gamma \left( x \right) := \left \{ 
\begin{array}{lll}
\min \left\lbrace p\left( x\right): \ \text{for all planes} \ p \geq u^{+} \ \text{in} \ B_{3} \right\rbrace , & \text{ in } & B_{3} \\
\\
0 & \text{ in } & \mathbb{R}^{n}\setminus B_{3}. 
\end{array}
\right. 
$$  
 
\begin{lemma} \label{cov. 1 lemma} Let $u \leq 0$ in $\mathbb{R}^{n}\setminus B_{1}$ and $\Gamma$ be its concave envelope. Suppose $M^{+}u\left( x \right) \geq -f\left( x\right)$ in $B_{1}$. Let $\rho_{0} = \rho_{0}\left( n \right) > 0$,
$$
 r_{k} := \rho_{0}2^{-\frac{1}{q_{\max}} }2^{-\mathfrak{C}\left( n + \sigma_{\min} \right)k},
$$
where 
$$q_{i} := -1 + \frac{3}{n+\sigma_{i}} + \sum\limits_{j\neq i}\frac{1}{n+\sigma_{j}}$$ 
and $q_{\max}:= \max \left\lbrace q_{1}, \dots, q_{n} \right\rbrace$.
Given $M>0$, we define
$$W_{k}\left( x \right)  := \Theta_{r_{k}}\setminus \Theta_{r_{k+1}} \cap $$
$$\cap \left\lbrace y : u\left( x + y\right) < u\left( x \right) + \langle y, \nabla \Gamma \left( x \right)\rangle - M\inf \limits_{z \in \Theta_{r_{k}}\setminus \Theta_{r_{k+1}}}\langle Az, z  \rangle  \right\rbrace,$$
 where the matrix $A = \left( a_{ij}\right)$ is defined by
$$
a_{ij}:= \left \{ 
\begin{array}{lll}
1, \ \text{ if } \ i=j=i_{\min}  \\
\\
0, \ \text{ if } \ i\neq j\\
\\
2^{\left( -\frac{1}{n+\sigma_{\min}} + \frac{1}{n+\sigma_{j}}\right) \frac{2}{q_{\max}} },\ \text{ if } \ i=j\neq i_{\min}.
\end{array}
\right.
$$
Then there exists a constant $C_{0}>0$, depending only on $n$, $\lambda$ (but not $\sigma_{i}$), such that, for any $x \in \left\lbrace u = \Gamma \right\rbrace $ and any $M>0$, there is a $k$ such that
\begin{equation}
\label{meas est 1}
\left | W_{k}\left( x \right)  \right | \leq C_{0} \frac{f\left( x \right)}{M} \left | \Theta_{r_{k}}\setminus \Theta_{r_{k+1}} \right |. 
\end{equation}
\end{lemma}
 
\begin{proof}
Notice that $u$ is touched by the plane  
$$
\Gamma \left( x \right) + \langle y - x, \nabla \Gamma \left( x \right) \rangle 
$$
from above at $x$. Then, from Lemma \ref{clas sense max}, $M^{+} u \left( x \right)$ is defined classically and we have 
\begin{equation}
\label{max exp}
M^{+} u \left( x \right) = c_{\sigma}  \int_{\mathbb{R}^{n}} \dfrac{\Lambda \delta^{+} - \lambda \delta^{-}}{\sum_{i=1}^{n}|y_{i}|^{n+\sigma_{i}}} dy.
\end{equation}
We will show that 
\begin{equation}
\label{non posit delta}
\delta \left( y \right) = u \left( x + y \right) + u \left( x - y \right) - 2u \left( x \right) \leq 0.
\end{equation}
In fact, since $u\left( x \right) = \Gamma \left( x \right) \geq 0$, we conclude that $\delta \left( y \right) \leq 0 $ whenever $u\left( x + y \right) \leq 0 $ and $u\left( x - y \right) \leq 0 $. Now suppose that $u\left( x + y \right) > 0 $. Then we have $x + y \in B_{1} \subset B_{3}$. Thus, from the definition of $\Gamma$, we find  
$$
u \left( x + y \right) - u \left( x \right) \leq 0.  
$$
Moreover, we have
$$
u\left( x - y \right) - u\left( x \right) \leq \left \{ 
\begin{array}{lll}
0, \ \text{ if } \ u\left( x - y \right) >0 \ \left( \text{then} \  x - y \in B_{1}\right)   \\
0, \ \text{ if } \ u\left( x - y \right) \leq 0 \ \left( - u\left( x\right) \leq 0 \right).		
\end{array}
\right.
$$
Thus, we obtain
$$
\delta \left( y \right) = \left(  u \left( x + y \right) - u \left( x \right) \right)  + \left( u \left( x - y \right) - u \left( x \right)\right)  \leq 0.  
$$
The case $u\left( x - y \right) > 0 $ is analogous to the case $u\left( x + y \right) > 0 $ and the inequality \eqref{non posit delta} is proved. Then, combining \eqref{max exp} and \eqref{non posit delta}, we find
\begin{eqnarray}
\label{cov. 1 lemma est1}
- f\left( x \right) & \leq & M^{+} u \left( x \right) \nonumber\\ 
& = & c_{\sigma} \int_{\Theta_{r_{0}}} \dfrac{ - \lambda \delta^{-} }{\sum_{i=1}^{n}|y_{i}|^{n+\sigma_{i}}} dy, 
\end{eqnarray}
where $r_{0} = \rho_{0} 2^{-\frac{1}{q_{\max}}}$. Since $x \in \left\lbrace u = \Gamma \right\rbrace$, we would like to emphasize that $ y \in W_{k}\left( x \right) $ implies $-y \in W_{k}\left( x \right)$. Thus, we find
\begin{equation}
\label{cov. 1 lemma est3}
W_{k}\left( x \right) \subset \Theta_{r_{k}} \setminus \Theta_{r_{k+1}}\cap \left\lbrace y : - \delta \left( y \right) > 2 M\inf \limits_{z \in \Theta_{r_{k}}\setminus \Theta_{r_{k+1}}}\langle Az, z  \rangle  \right\rbrace.
\end{equation}
Using \eqref{cov. 1 lemma est1}, we estimate
\begin{eqnarray}
\label{cov. 1 lemma est2}
f\left( x \right) & \geq & c\left(n, \lambda\right)\left[  c_{\sigma}  \sum_{k=1}^{\infty}\int_{\Theta_{r_{k}}\setminus \Theta_{r_{k+1}}} \dfrac{  \delta^{-} }{\sum_{i=1}^{n}|y_{i}|^{n+\sigma_{i}}}dy \right]  \nonumber \\ 
& \geq &  c\left(n, \lambda\right) \sum_{k=1}^{\infty}\left[ c_{\sigma}r^{-1}_{k}\int_{W_{k}} \delta^{-} dy \right] . 
\end{eqnarray}
Let us assume by contradiction that \eqref{meas est 1} is not valid. Then, using \eqref{cov. 1 lemma est3} and \eqref{cov. 1 lemma est2}, we obtain
$$
f\left( x \right)  \geq  \nonumber  c\left(n, \lambda\right) \left[ c_{\sigma} \sum_{k=1}^{\infty} \left( M 2^{ \left( -\frac{2}{n+\sigma_{\min}}\right)\frac{1}{q_{\max}} }\sum_{i=1}^{n} 2^{-\frac{2\left( n + \sigma_{\min} \right)}{n + \sigma_{i}}k}\right) \right.$$
$$\hspace*{7cm} \left. \dfrac{C_{0}f\left( x \right)r_{k}^{-1} \left |\Theta_{r_{k}} \setminus \Theta_{r_{k+1}} \right |}{M}\right]$$ 
$$\geq  c_{1}C_{0}f\left( x \right)\left[ c_{\sigma} 2^{ \left( -\frac{2}{n+\sigma_{\min}}\right)\frac{1}{q_{\max}} }\sum_{k=1}^{\infty} \sum_{i=1}^{n} \left(  2^{-\frac{2\left( n + \sigma_{\min} \right)}{n + \sigma_{i}}k} \left( r^{-1}_{k}\prod_{j=1}^{n} r^{\frac{1}{n+\sigma_{j}}}_{k} \right)\right)  \right]$$
$$= 2^{-1}c_{2}C_{0}f\left( x \right)\left[ c_{\sigma}\sum\limits_{i=1}^{n} \left( \sum\limits_{k=1}^{\infty} 2^{ -\mathfrak{C}\left( n + \sigma_{\min} \right) q_{i}k }\right) \right]. $$
Then, we get
\begin{eqnarray*}
f\left( x \right) & \geq & c_{3} C_{0} f\left( x \right)\left[ \sum\limits_{i=1}^{n} \left( c_{\sigma} \sum\limits_{i=1}^{\infty}2^{-\mathfrak{C}\left( n + \sigma_{\min} \right) q_{i}k} \right) \right]\\ 
& = & c_{3} C_{0}f\left( x \right)\sum\limits_{i=1}^{n}\dfrac{c_{\sigma}}{1- 2^{-\mathfrak{C}\left( n + \sigma_{\min} \right) q_{i}}}\\ 
& \geq & \dfrac{c_{3} C_{0}c_{\sigma}f\left( x \right)}{1- 2^{-\mathfrak{C}\left( n + \sigma_{\min} \right) c_{\sigma}}}.
\end{eqnarray*}
Finally, since $\frac{ c_{\sigma}}{1- 2^{-\mathfrak{C}\left( n + \sigma_{\min} \right) c_{\sigma}}}$ is bounded away from zero, for all $\sigma_{i} \in \left( 0, 2\right)$, we find
$$
f\left( x \right)  \geq  c_{4}\left(n, \lambda \right)C_{0}f\left( x \right),
$$
which is a contradiction if $C_{0}$ is chosen large enough.
\end{proof}

\begin{remark}
In the proof of Lemma \ref{cov. 1 lemma} we have used the matrix $A:= \left( a_{ij}\right)$ to control the term $2^{-\frac{1}{q_{\max}}} $, which can degenerate. This term corresponds to the factor $2^{-\frac{1}{2-\sigma}}$ in the isotropic nonlocal ABP estimate in \cite{CS}. We also emphasise that the matrix $A$ is diagonal, has norm one and, if $\sigma_{i}= \sigma$, we obtain the matrix for the isotropic case $A=Id$. 
\end{remark}
The following result is a direct consequence of the arguments used in the proof of \cite[Lemma 8.4]{CS}.

\begin{lemma} \label{cov. 2 lemma, SC} Let $\Gamma$ be a concave function in $B_{1}$ and $v\in \mathbb{R}^{n}$. Assume that, for a small $\varepsilon > 0$,
$$
\left |\left(  B_{1} \setminus B_{\frac{1}{2}} \right) \cap \left\lbrace y : \Gamma \left( y\right) < \Gamma \left( 0 \right) + \langle T\left( y \right), v \rangle - h \right\rbrace \right | \leq \varepsilon \left | B_{1} \setminus B_{\frac{1}{2}} \right |,
$$
where $T:\mathbb{R}^{n} \rightarrow \mathbb{R}^{n}$ is a linear map. Then 
$$\Gamma \left(  y \right) \geq \Gamma \left( 0 \right) + \langle T\left( y \right), v \rangle - h$$ 
in the whole ball $B_{\frac{1}{2}}$.
\end{lemma}
\begin{proof}
Let $y \in B_{\frac{1}{2}}$. There exist $B_{\frac{1}{2}}\left( y_{1}\right) \subset B_{1} \setminus B_{1/2}$ and $B_{\frac{1}{2}}\left( y_{2}\right) \subset B_{1} \setminus B_{1/2}$ such that 
$$
L\left( B_{\frac{1}{2}}\left( y_{1}\right) \right) = B_{\frac{1}{2}}\left( y_{2}\right),
$$
where $L: B_{\frac{1}{2}}\left( y_{1}\right) \rightarrow  B_{\frac{1}{2}}\left( y_{2}\right)$ is the linear map
$$
L\left( z \right) = 2y - z.
$$
Geometrically, the balls $B_{\frac{1}{2}}\left( y_{1}\right)$ and $B_{\frac{1}{2}}\left( y_{2}\right)$ are symmetrical with respect to $y$. Then, if $\varepsilon > 0$ is sufficiently small, there will be two points $z_{1} \in B_{\frac{1}{2}}\left( y_{1}\right)$ and $z_{2} \in B_{\frac{1}{2}}\left( y_{2}\right)$ such that 
\begin{enumerate}
\item $y = \dfrac{z_{1}+z_{2}}{2}$;
\medskip
\item $\Gamma\left( z_{1} \right) \geq \Gamma\left( 0 \right) + \langle T\left( z_{1} \right), v \rangle - h$;
\medskip
\item $ \Gamma\left( z_{2} \right) \geq \Gamma\left( 0 \right) + \langle T\left( z_{2} \right), v \rangle - h$. 
\medskip
\end{enumerate} 
Hence, since $T$ and $\langle \cdot, v \rangle$ are linear maps and $\Gamma$ is a concave function, we obtain
$$
\Gamma\left( y \right) \geq \Gamma\left( 0 \right) + \langle T\left( y \right), v \rangle - h.
$$
\end{proof}

Using Lemma \ref{cov. 2 lemma, SC}, we will prove the version of Lemma 8.4 in \cite{CS} for our problem.

\begin{lemma}
\label{cov. 2 lemma}
Let $r>0$ and $\Gamma$ be a concave function in $E_{r, \frac{1}{2}}$. There exists $\varepsilon_{0} > 0$ such that if
$$
\left |E_{r,\frac{1}{2}} \setminus E_{r, \frac{1}{4}} \cap \left\lbrace y : \Gamma \left( y\right) < \Gamma \left( 0 \right) + \langle y, \nabla \Gamma \left(0 \right)\rangle - h \right\rbrace \right | \leq \varepsilon \left |  E_{r,\frac{1}{2}} \setminus E_{r, \frac{1}{4}} \right |,
$$
for $0< \varepsilon \leq \varepsilon_{0}$, then 
$$\Gamma \left(  y \right) \geq \Gamma \left( 0 \right) + \langle y, \nabla \Gamma \left( 0 \right)\rangle - h$$ 
in the whole set $E_{r, \frac{1}{4}}$.
\end{lemma}

\begin{proof}
Let $T:\mathbb{R}^{n} \rightarrow \mathbb{R}^{n}$ be the linear map defined by
$$
Te_{i}= \frac{r^{\frac{1}{n+\sigma_{i}}}}{2}e_{i},
$$
where $e_{i}$ denotes the $i$-th vector of the canonical basis of $\mathbb{R}^{n}$. If 
$$
A:= \left(  B_{1} \setminus B_{\frac{1}{2}}\right) \cap \left\lbrace y : \tilde{\Gamma} \left( y\right) < \tilde{\Gamma} \left( 0 \right) + \langle T\left( y \right), \nabla \Gamma \left( 0 \right) \rangle - h \right\rbrace 
$$
and
$$
D:= E_{r,\frac{1}{2}} \setminus E_{r, \frac{1}{4}} \cap \left\lbrace y : \Gamma \left( y\right) < \Gamma \left( 0 \right) + \langle y, \nabla \Gamma \left(0 \right)\rangle - h \right\rbrace,
$$
we have
$$
 A = T^{-1}\left( D \right),  
$$
where $\tilde{\Gamma}\left( x \right) := \Gamma \left(  T \left( x \right) \right)$. Moreover,
$$
 B_{1} \setminus B_{\frac{1}{2}}  = T^{-1}\left( E_{r,\frac{1}{2}} \setminus E_{r, \frac{1}{4}} \right) \quad \text{and} \quad   B_{\frac{1}{2}}  = T^{-1}\left( E_{r, \frac{1}{4}} \right). 
$$
Then, taking into account that $\tilde{\Gamma}$ is concave, the lemma follows from Lemma \ref{cov. 2 lemma, SC}.
\end{proof}

\begin{corollary}
\label{ABP cor 1}
Let $\varepsilon_{0} > 0$ be as in Lemma \ref{cov. 2 lemma}. Given $0 < \varepsilon \leq \varepsilon_{0}$, there exists a constant $C\left( n, \lambda, \varepsilon \right) > 0$ such that for any function $u$ satisfying the same hypothesis as in Lemma \ref{cov. 1 lemma}, there exist $r \in \left( 0, \rho_{0} 2^{-\frac{1}{q_{\max}} } \right)$ and $k= k\left( x \right)$ such that 
$$\left | \Theta_{r} \setminus \Theta_{s r} \cap \left\lbrace y : u\left( x + y\right) < u\left( x \right) + \langle y, \nabla \Gamma \left( x \right)\rangle - C f\left( x \right)\sum \limits_{i=1}^{n} r^{\frac{2}{n+\sigma_{i}}} \right\rbrace \right |   $$
\begin{equation} \label{ABP cor 1.1}
\leq \varepsilon_{1} \left | \Theta_{r} \setminus \Theta_{s r} \right |
\end{equation}
and 
$$
\left | \nabla \Gamma \left( R^{}_{a,s^{k + 1}}\left( x \right) \right) \right | \leq C f\left( x \right)^{n} \left | R^{}_{a,s^{k + 1}}\left( x \right) \right |,
$$
where $r = \rho_{0} 2^{-\frac{1}{q_{\max}} }2^{-\mathfrak{C}\left( n+\sigma_{\min} \right)k}$, $a = \rho_{0} 2^{-\frac{1}{q_{\max}}}$ and $s = 2^{-\mathfrak{C}\left( n+\sigma_{\min} \right)}$.
\end{corollary}
\begin{proof}
Taking $M=\frac{C_{0}}{\varepsilon}f\left( x \right)$ in Lemma \ref{cov. 1 lemma}, we obtain \eqref{ABP cor 1.1} with $C_{1}:= \frac{C_{0}}{\varepsilon}$. Moreover, since $u\left( x \right) = \Gamma\left( x \right)$ and $u\left( x + y\right) \leq \Gamma \left( x + y\right)$, for $y \in E_{r, \frac{1}{2}}$, we have
$$
E_{ r, \frac{1}{2}} \setminus E_{r, \frac{1}{4}} \cap \left\lbrace y : \Gamma\left( x + y\right) < u\left( x \right) + \langle y, \nabla \Gamma \left( x \right)\rangle - C_{1} f\left( x\right) \inf \limits_{z \in \Theta_{r}\setminus \Theta_{sr}}\langle Az, z  \rangle   \right\rbrace$$
$$\subset W_{r}\left( x \right) 
$$
where
$$W_{r}\left( x \right) := \Theta_{r} \setminus \Theta_{s r} \cap $$
$$\cap  \left\lbrace y : u\left( x + y\right) < u\left( x \right) + \langle y, \nabla \Gamma \left( x \right)\rangle - C_{1} f\left( x\right) \inf \limits_{z \in \Theta_{r}\setminus \Theta_{sr}}\langle Az, z  \rangle  \right\rbrace.$$
Then, from Lemma \ref{cov. 2 lemma} and the concavity of $\Gamma$, we find
$$
 0 \leq F\left( y \right) \leq  2C_{1}f\left( x \right) \inf \limits_{z \in \Theta_{r}\setminus \Theta_{sr}}\langle Az, z  \rangle  \quad \ \text{in} \ E_{r, \frac{1}{4}}, 
$$
where 
$$F\left( y \right):= \Gamma \left( x + y \right) - \Gamma \left( x \right) - \langle y,  \nabla \Gamma \left( x \right) \rangle + C_{1}f\left( x \right)\inf \limits_{z \in \Theta_{r}\setminus \Theta_{sr}}\langle Az, z  \rangle .$$ 
Notice that
$$
\nabla F\left( x + y \right) = \nabla \Gamma\left( x + y \right) - \nabla \Gamma \left( x \right) .
$$
Then, since $F$ is concave, we obtain 
\begin{eqnarray*}
\left | \nabla \Gamma\left( x + y \right) - \nabla \Gamma \left( x \right) \right | & \leq & \dfrac{\Vert F \Vert_{L^{\infty}\left(E_{r, \frac{1}{4}} \right)}}{\dist \left( E_{r, \frac{1}{4}} , E_{ r, \frac{1}{8}}\right)} \\ 
& \leq & \dfrac{C_{1} f\left( x\right) \inf \limits_{z \in \Theta_{r}\setminus \Theta_{sr}}\langle Az, z  \rangle }{\dist \left( E_{r, \frac{1}{4}} , E_{ r, \frac{1}{8}}\right)}  \\ 
& \leq & C_{2}f\left( x \right)r^{\frac{1}{n + \sigma_{\min}}}.
\end{eqnarray*}
Thus, we have
$$
\nabla \Gamma \left( E_{ r, \frac{1}{8}} \right) \subset B_{C_{2}f\left( x \right)r^{\frac{1}{n + \sigma_{\min}}}}\left( \nabla \Gamma \left( x \right) \right)
$$
and obtain
$$
\left | \nabla \Gamma \left( R_{a, s^{k + 1} } \right) \right | \leq \left | \nabla \Gamma \left( E_{s r, \frac{1}{8}} \right) \right | \leq C_{3}f\left( x \right)^{n} \left | R_{a, s^{k+1}} \right |.    
$$
Finally, taking $C=\max \left\lbrace C_{1}, C_{3} \right\rbrace $, the lemma is proven.
\end{proof}
 
The following covering lemma is a fundamental tool in our analysis.

\begin{lemma}[Covering Lemma, {\cite[Lemma 3]{CCal}}]
\label{covering lemma}
Let $S$ be a bounded subset of $\mathbb{R}^{n}$ such that for each $x \in S$ there exists an $n$-dimensional rectangle $\mathcal{R}\left( x \right)$, centered at $x$, such that:
\begin{itemize}
\item the edges of $\mathcal{R}\left( x \right)$ are parallel to the coordinate axes;
\medskip
\item the length of the edge of $\mathcal{R}\left( x \right)$ corresponding to the $i$-th axis is given by $h_{i}\left(t \right)$, where $t=t\left( x \right)$, $h_{i}\left( t \right)$ is an increasing function of the parameter $t\geq 0$, continuous at $t=0$, and $h_{i}\left( 0 \right)=0$. 
\end{itemize}
Then there exist points $\left\lbrace x_{k} \right\rbrace$ in $S$ such that 
\begin{enumerate}
\item $S \subset \bigcup_{k=1}^{\infty} \mathcal{R}\left( x_{k}\right)$;
\medskip
\item each $x \in S$ belongs to at most $C=C\left( n \right) >0$ different rectangles. 
\end{enumerate}
\end{lemma}

The Corollary \ref{ABP cor 1} and the Covering Lemma \ref{covering lemma} allow us to obtain a lower bound on the volume of the union of the level sets $\Theta_{r}$ where $\Gamma$ and $u$ detach quadratically from the corresponding tangent planes to $\Gamma$ by the volume of the image of the gradient map, as in the standard ABP estimate.
\begin{corollary}
\label{ABP cor 2}
For each $x \in \Sigma$, let $\Theta_{r}\left( x \right)$ be the level set obtained in Corollary \ref{ABP cor 1}. Then, we have
$$
C\left( \sup \limits u \right)^{n} \leq \left | \bigcup \limits_{x \in \Sigma} \Theta_{r}\left( x \right)\right |.
$$
\end{corollary}

The nonlocal anisotropic version of the ABP estimate now reads as follows.

\begin{theorem}
\label{ABP Nonlocal theorem}
Let $u$ and $\Gamma$ be as in Lemma \ref{cov. 1 lemma}. There is a finite family of open rectangles $\left\lbrace \mathcal{R}_{j}\right\rbrace_{j \in \left\lbrace 1, \dots , m \right\rbrace }$ with diameters $d_{j}$ such that the following hold:
\begin{enumerate}

\item Any two rectangles $\mathcal{R}_{i}$ and $\mathcal{R}_{j}$ in the family do not intersect.

\medskip

\item $\left\lbrace  u = \Gamma \right\rbrace \subset \bigcup_{j=1}^{m} \overline{\mathcal{R}}_{j} $.

\medskip

\item $\left\lbrace  u = \Gamma \right\rbrace \cap \overline{\mathcal{R}}_{j} \neq \emptyset $ for any $\mathcal{R}_{j}$.

\medskip

\item $d_{j} \leq  \sqrt{\sum \limits_{i=1}^{n}\left( \rho_{0} 2^{-\frac{1}{q_{\max}}}\right)^{\frac{2}{n+\sigma_{i}}} }$.

\medskip

\item $ \left | \nabla \Gamma \left( \overline{\mathcal{R}}_{j} \right) \right | \leq C \left( \max_{\overline{\mathcal{R}}_{j}} f^{+} \right)^{n} \left | \overline{\mathcal{R}}_{j} \right |$.

\medskip

\item $ \left | \left\lbrace y \in C\tilde{\mathcal{R}}_{j} : u\left( y \right) \geq \Gamma \left( y \right) - C\left( \max_{\overline{\mathcal{R}}_{j}} f \right) \left( \tilde{d}_{j}\right)^{2} \right\rbrace \right | \geq \varsigma \left | \tilde{\mathcal{R}}_{j} \right |$,

\end{enumerate} 
where $\tilde{d}_{j}$ is the diameter of the rectangle $\tilde{\mathcal{R}}_{j}$ corresponding to $\mathcal{R}_{j}$. The constants $\varsigma > 0$ and $C >0$ depend only on $n$, $\lambda$ and $\Lambda$.
\end{theorem}
\begin{proof}
We cover the ball $B_{1}$ with a tiling of rectangles of edges 
$$\dfrac{\left( \rho_{0} 2^{-\frac{1}{q_{\max}}}\right)^{\frac{1}{n+\sigma_{i}}}}{2^{-\mathfrak{C}}}.$$
We discard all those that do not intersect $\left\lbrace  u = \Gamma \right\rbrace$. Whenever a rectangle does not satisfy (5) and (6), we split its edges by $2^{n\mathfrak{C}}$ and discard those whose closure does not intersect $\left\lbrace  u = \Gamma \right\rbrace$. Now we prove that all remaining rectangles satisfy (5) and (6) and that this process stops after a finite number of steps.

As in \cite{CS} we will argue by contradiction. Suppose the process is infinite. Thus, there is a sequence of nested rectangles $\mathcal{R}_{j}$ such that the intersection of their closures will be a point $x_{0}$. Moreover, since 
$$
\left\lbrace  u = \Gamma \right\rbrace \cap \overline{\mathcal{R}}_{j} \neq \emptyset 
$$ 
and $\left\lbrace  u = \Gamma \right\rbrace$ is closed, we have $x_{0} \in \left\lbrace  u = \Gamma \right\rbrace$. Let $0 < \varepsilon_{1} < \varepsilon_{0}$, where $\varepsilon_{0}$ is as in Lemma \ref{ABP cor 1}. Then, there exist 
$$r \in \left( 0, \rho_{0} 2^{-\frac{1}{q_{\max}} } \right)$$ 
and $k_{0}= k_{0}\left( x_{0} \right)$ such that 
$$\left | \Theta_{r} \setminus \Theta_{s r} \cap \left\lbrace y : u\left( x + y\right) < u\left( x \right) + \langle y, \nabla \Gamma \left( x \right)\rangle - C f\left( x \right)\sum \limits_{i=1}^{n} r^{\frac{2}{n+\sigma_{i}}} \right\rbrace \right | 
$$
\begin{equation}\label{meas est diam}
\leq \varepsilon_{1}  \left | \Theta_{r} \setminus \Theta_{s r} \right |
\end{equation}
and 
\begin{eqnarray}
\label{meas est diam 2} 
\left | \nabla \Gamma \left( R^{}_{a,s^{k_{0} + 1}}\left( x_{0} \right) \right) \right | \leq C f\left( x_{0} \right)^{n} \left |  R^{}_{a,s^{k_{0} + 1}}\left( x_{0} \right)\right |,
\end{eqnarray}
where 
$$r = \rho_{0} 2^{-\frac{1}{q_{\max}}} 2^{-\mathfrak{C}\left( n + \sigma_{\min}\right) k_{0}}.$$
Let $\mathcal{R}_{j}$ be the largest rectangle in the family containing $x_{0}$ and contained in $R^{}_{a,s^{k_{0} + 1}}\left( x_{0} \right)$. Then $x_{0} \in \mathcal{R}_{j}$ and $\mathcal{R}_{j}$ has edges $l_{i}$ satisfying 
$$
2^{-\mathfrak{C}\left( k_{0} + 2\right)} \left( \rho_{0} 2^{-\frac{1}{q_{\max}}} \right)^{\frac{1}{n+\sigma_{i}}} \leq l_{i} < 2^{-\mathfrak{C}\left( k_{0} + 1\right)} \left( \rho_{0} 2^{-\frac{1}{q_{\max}}} \right) ^{\frac{1}{n+\sigma_{i}}}. 
$$
Thus, we get
$$
\mathcal{R}_{j} \subset R^{}_{a,s^{k_{0} + 1}} \quad \text{and} \quad  \Theta_{r} \subset C\tilde{\mathcal{R}}_{j},
$$
for some $C=C\left( n \right) > 1$. Furthermore, since $\Gamma$ is concave in $B_{2}$, we find
$$
\Gamma \left( y \right) \leq u \left( x_{0} \right) + \langle y - x_{0}, \nabla \Gamma \left( x_{0} \right)\rangle 
$$
in $B_{2}$. Thus, denoting
$$
A_{j}:= \left\lbrace y \in C \tilde{\mathcal{R}}_{j} : u\left( y \right) \geq \Gamma \left( y \right) - C\left( \max_{\overline{\mathcal{R}}_{j}} f \right) \left( \tilde{d}_{j}\right)^{2}  \right\rbrace ,
$$
using \eqref{meas est diam}, \eqref{meas est diam 2} and that $l_{i}$ and $s^{-k_{0}} \left( \rho_{0} 2^{-\frac{1}{q_{\max}}} \right) ^{\frac{1}{n+\sigma_{i}}}$ are comparable, we obtain
\begin{eqnarray*}
\left | A_{j} \right | & \geq & \left | \left\lbrace y  \in C \tilde{\mathcal{R}}_{j} : u\left( y\right) \geq u\left( x_{0} \right) + \langle y - x_{0}, \nabla \Gamma \left( x_{0} \right)\rangle \right. \right. \\
& & \left. \left. - C f\left( x_{0} \right)\sum \limits_{i=1}^{n} r^{\frac{2}{n+\sigma_{i}}} \right\rbrace \right | \\ 
& \geq & \left( 1 - \varepsilon_{1} \right) \left | \Theta_{r} \setminus \Theta_{s r} \right | \\ 
& \geq & \varsigma \left |\tilde{\mathcal{R}}_{j}\right |
\end{eqnarray*}
and
\begin{eqnarray*}
\left | \nabla \Gamma \left( \mathcal{R}_{j} \right) \right | & \leq & \left | \nabla \Gamma \left( R^{}_{a,s^{k_{0} + 1}}\left( x_{0} \right) \right) \right | \\  
& \leq & C f\left( x_{0} \right)^{n} \left |   R^{}_{a,s^{k_{0} + 1}}\left( x_{0} \right) \right | \\ 
& = & C_{1} f\left( x_{0} \right)^{n} \left | \mathcal{R}_{j} \right |.
\end{eqnarray*} 
Then $\mathcal{R}_{j}$ would not be split and the process must stop, which is a contradiction.
\end{proof}

\section{A barrier function}\label{Barrier function section}

With the aim of localising the contact set of a solution $u$ of the maximal equation, as in Lemma \ref{cov. 1 lemma}, we build a barrier function which is a supersolution of the minimal equation outside a small ellipse and is positive outside a large ellipse. 

\begin{lemma}
\label{Barr function 1}
Given $R > 1, $ there exist $p >0$ and $\sigma_{0} \in \left( 0, 2\right)$ such that the function
$$f\left( x \right) = \min \left( 2^{p}, \  | x |^{-p}\right) $$
satisfies
$$M^{-}f\left( x \right) \geq 0, $$
for $\sigma_{0} < \sigma_{\min}$ and $1 \leq | x | \leq R $, where $p = p\left( n, \lambda, \Lambda, R \right)$, $\sigma_{0} = \sigma_{0}\left( n, \lambda, \Lambda, R \right)$.
\end{lemma}
    
\begin{proof}
In the sequel we will use the following elementary inequalities:
\begin{equation}
\label{bar func: elem ineq 1}
\left( a_{2} + a_{1}\right)^{-s} + \left( a_{2} - a_{1}\right)^{-s} \geq  2a^{-s}_{2} + s\left( s + 1\right)a^{2}_{1}a_{2}^{-s-2} 
\end{equation}
and
\begin{equation}
\label{bar func: elem ineq 2}
\left( a_{2} + a_{1}\right)^{-s} \geq a^{-s}_{2}\left( 1 - s\frac{a_{1}}{ a_{2}}\right). 
\end{equation}
where $0 < a_{1} < a_{2}$ and $s > 0$.
Taking into account the inequalities \eqref{bar func: elem ineq 1} and \eqref{bar func: elem ineq 2}, we estimate, for $| y | < \frac{1}{2}$, 
\begin{eqnarray*}
\delta (f,e_1,y)& := & \nonumber | e_{1} + y|^{-p} + | e_{1} - y|^{-p} - 2 \\ 
& = & \left(  1 + |y|^{2} + 2y_{1} \right)^{-\frac{p}{2}} +  \left(  1 + |y|^{2} - 2y_{1} \right)^{-\frac{p}{2}} - 2 \\ 
& \geq & \left( 1 + |y|^{2} \right)^{-\frac{p}{2}} + p\left( p + 2 \right)y_{1}^{2} \left( 1 + |y|^{2} \right)^{-\frac{p+4}{2}} - 2 \\ 
& \geq &  2\left( 1 - \frac{p}{2}|y|^{2}\right)  +  p\left( p + 2 \right)^{2}y_{1}^{2} - p\left( p + 4\right) \frac{\left( p + 2 \right)}{2}y^{2}_{1}|y|^{2}  - 2\\
& = & p \left[  - |y|^{2} + \left( p + 2 \right)y_{1}^{2} - \left( p + 4\right) \frac{\left( p + 2 \right)}{2}y^{2}_{1}|y|^{2}\right].
\end{eqnarray*}
Given $1 \leq |x| \leq R $, there is a rotation $T_{x}:\mathbb{R}^{n} \rightarrow \mathbb{R}^{n}$ such that $x = |x| Te_{1}$. Thus, changing variables, we get
$$
M^{-}f\left( x \right) = c_{\sigma} |x|^{n-p} \left | \det T_{x} \right | \left[ \int_{\mathbb{R}^{n}} \dfrac{\lambda \delta^{+}\left(f, e_{1}, y \right) - \Lambda\delta^{-}\left(f, e_{1}, y \right)}{\sum_{i=1}^{n}|\left( |x|T_{x}y\right) _{i}|^{n+\sigma_{i}}} dy \right]. 
$$
Then, we estimate
\begin{eqnarray}\label{below estimate for min sol}
|x|^{p-n} M^{-}f\left( x \right) & = & \nonumber c_{\sigma} \int_{B_{1/4}\left( 0 \right)}  \dfrac{\Lambda \delta^{+}\left(f, e_{1}, y \right) - \lambda \delta^{-}\left(f, e_{1}, y \right) }{\sum_{i=1}^{n}||x|\left( T_{x}y\right) _{i}|^{n+\sigma_{i}}} dy \nonumber \\ 
& &+  c_{\sigma} \int_{\mathbb{R}^{n} \setminus B_{1/4}\left( 0 \right)} \dfrac{\Lambda \delta^{+}\left(f, e_{1}, y \right) - \lambda \delta^{-}\left(f, e_{1}, y \right) }{\sum_{i=1}^{n}||x|\left( T_{x}y\right) _{i}|^{n+\sigma_{i}}} dy \nonumber \\ 
& \geq &  c_{\sigma} \int_{B_{1/4}\left( 0 \right)}\dfrac{2p \lambda\left( p + 2 \right)y^{2}_{1}}{\sum_{i=1}^{n}||x|\left( T_{x}y\right) _{i}|^{n+\sigma_{i}}} dy  \nonumber\\
& & -c_{\sigma} \int_{B_{1/4}\left( 0 \right)}\dfrac{2p \Lambda |y|^{2}}{\sum_{i=1}^{n}||x|\left( Ty\right) _{i}|^{n+\sigma_{i}}} dy \nonumber\\ 
&  & - c_{\sigma}\int_{B_{1/4}\left( 0 \right)}\dfrac{\frac{1}{2}p\left( p + 4\right) \left( p + 2 \right)|y|^{4}}{\sum_{i=1}^{n}||x|\left( T_{x}y\right) _{i}|^{n+\sigma_{i}}} dy \nonumber\\ 
&  &+ c_{\sigma} \int_{\mathbb{R}^{n} \setminus B_{1/4}\left( 0 \right)} \dfrac{- \lambda 2^{p+1} }{\sum_{i=1}^{n}||x|\left( T_{x}y\right) _{i}|^{n+\sigma_{i}}|} dy \nonumber\\ 
& := & I_{1} + I_{2} + I_{3} + I_{4},
\end{eqnarray}
where $I_{1}$, $I_{2}$, $I_{3}$ and $I_{4}$ represent the three terms on the right-hand side of the above inequality. 

We estimate
\begin{eqnarray*}
\label{I_{1} estimate}
p^{-1} I_{1} & \geq & n^{-1}c_{\sigma}\lambda \left( p+2\right) |x|^{- \left( n+2\right)} \int_{B_{1/4}\left( 0 \right)}\dfrac{y_{1}^{2}}{|y|^{n+\sigma_{\min}}} dy  \\ 
& \geq & R^{- \left( n+2\right)}n^{-1} \left[ c_{\sigma}\lambda \left( p+2\right)\int_{\partial B_{1}} y^{2}_{1} d\nu \left( y \right)\right] \int_{0}^{\delta/4}t^{1-\sigma_{\min}} dt \\ 
& \geq & C_{3}\dfrac{c_{\sigma}}{2-\sigma_{\min}}\left[\left( p+2\right)\int_{\partial B_{1}} y^{2}_{1} d\nu \left( y \right)\right]\left(\dfrac{1}{4} \right)^{2-\sigma_{\min}} \\ 
& \geq & C_{3}c\left( n \right) \left[\left( p+2\right)\int_{\partial B_{1}} y^{2}_{1} d\nu \left( y \right)\right] ,
\end{eqnarray*}
where $C_{3}=C_{3}\left( n, \lambda, \Lambda, R \right)>0$. Moreover, if $C=C\left( n\right) >0$ is a positive constant such that $B_{1/4}\left( 0 \right)  \subset \Theta_{C}$, we have, for $|x| \geq 1$,
\begin{eqnarray*}
\label{I_{2} estimate}
p^{-1}I_{2} & \geq & - C_{4} c_{\sigma} \int_{B_{1/4}\left( 0 \right)}\dfrac{|y|^{2}}{\sum_{i=1}^{n}|\left( T_{x}y\right) _{i}|^{n+\sigma_{i}}} dy \\ 
& = & - C_{4} c_{\sigma}\left | \det T^{-1}_{x} \right |\int_{B_{1/4}\left( 0 \right)}\dfrac{|T_{x}^{-1}y|^{2}}{\sum_{i=1}^{n}|y_{i}|^{n+\sigma_{i}}} dy \\ 
& = & - C_{4} c_{\sigma}\int_{B_{1/4}\left( 0 \right)}\dfrac{|y|^{2}}{\sum_{i=1}^{n}|y_{i}|^{n+\sigma_{i}}} dy \\ 
& \geq & - C_{4} c_{\sigma}\int_{\Theta_{C}}\dfrac{|y|^{2}}{\sum_{i=1}^{n}|y_{i}|^{n+\sigma_{i}}} dy, 
\end{eqnarray*}
where $C_{4}=C_{4}\left( n, \lambda, \Lambda \right)$.
We have also
$$
c_{\sigma}\int_{\Theta_{C}}\dfrac{|y|^{2}}{\sum_{i=1}^{n}| y_{i}|^{n+\sigma_{i}}} dy = c_{\sigma}\sum_{k=1}^{\infty} \int_{\Theta_{r_{k}} \setminus \Theta_{r_{k+1}}}\dfrac{|y|^{2}}{\sum_{i=1}^{n}|y_{i}|^{n+\sigma_{i}}} dy \leq C_{5},
$$
where $r_{k}:= C 2^{-k}$ and $C_{5}=C_{5}\left( n, \lambda, \Lambda \right)$.
Moreover, using the elementary inequality
$$
\left( a + b \right)^{m}  \leq 2^{m}\left( a^{m} + b^{m}\right), \quad \ \text{for all}, \ a,b,m \in \left( 0, \infty \right) ,
$$
we get
\begin{eqnarray}
\label{I_{3} estimate}
I_{3} & \geq & - C_{6}c_{\sigma} 2^{\frac{n+\sigma_{\max}}{2}}\int_{B_{\delta/4}\left( 0 \right)}\dfrac{|y|^{4}}{|y|^{n+\sigma_{\max}}} dy \nonumber \\ 
& \geq & - C_{7}\dfrac{c_{\sigma}}{\left( 4 - \sigma_{\max}\right)}\left( \dfrac{ 1}{4}\right) ^{4-\sigma_{\max}}
\end{eqnarray}
and 
\begin{eqnarray}
\label{I_{4} estimate}
I_{4} & \geq & -  c_{\sigma}\left( \dfrac{1}{4}\right)^{-\sigma_{\max}+\sigma_{\min}} 2^{\frac{n+\sigma_{\min}}{2}}\int_{\mathbb{R}^{n} \setminus B_{1/4}\left( 0 \right)} \dfrac{ \Lambda 2^{p+2} }{|y|^{n+\sigma_{\min}}} dy \nonumber \\ 
& = & -  C_{8}\left( \dfrac{1}{4}\right)^{-\sigma_{\max}} 2^{\frac{n+\sigma_{\min}}{2}} \dfrac{c_{\sigma}}{\sigma_{\min}} \nonumber \\ 
& \geq & - C_{8} \left( \dfrac{1}{4}\right)^{-\sigma_{\max}} \dfrac{c_{\sigma}}{\sigma_{\min}}, 
\end{eqnarray}
for positive constants $C_{7}=C_{7}\left( n, \lambda, \Lambda,p \right)$ and $C_{8}=C_{8}\left( n, \lambda, \Lambda,p \right)$. Choosing $p=p\left( n, \lambda, \Lambda, R \right) > 0$ such that
$$
C_{3} \left( p+2\right) \int_{\partial B_{1}} y^{2}_{1} d\nu \left( y \right) - C_{4}C_{5}  > 0 
$$
and combining \eqref{below estimate for min sol}, \eqref{I_{3} estimate} and \eqref{I_{4} estimate}, there is a positive constant $\sigma_{0}=\sigma_{0}\left( n, \lambda, \Lambda, R \right) < 2 $ such that
$$
|x|^{p-n} M^{-}f\left( x \right)  \geq   C_{9} >0,
$$
for a positive constant $C_{9}=C_{9}\left( n, \lambda, \Lambda, R \right)$ and $\sigma_{0} < \sigma_{\min} < 2$.
\end{proof}
     
\begin{corollary}
\label{Coroll Barr function 1}
Given $r > 0 $, $\sigma_{0} \in \left( 0,  2\right)$, $\sigma_{0} < \sigma_{\min}$, and $R > 1$, there exist $s>0$ and $p >0$ such that the function
$$
f\left( x \right) = \min \left( s^{-p}, \  | x |^{-p}\right)  
$$
satisfies
$$
M^{-}f\left( x \right) \geq 0,
$$
for $1 \leq | x | \leq R $, where $p = p\left( n, \lambda, \Lambda,  R \right)$ and $s = s\left( n, \lambda, \Lambda,\sigma_{0}, R \right)$.
\end{corollary}

\begin{proof}
Since $c_{\sigma} \geq c\left( n \right)  \left( 2 - \sigma_{\min} \right)$ and 
$$
c_{\sigma}\int_{\Theta_{C}}\dfrac{|y|^{2}}{\sum_{i=1}^{n}| y_{i}|^{n+\sigma_{i}}} dy = c_{\sigma}\sum_{k=1}^{\infty} \int_{\Theta_{r_{k}} \setminus \Theta_{r_{k+1}}}\dfrac{|y|^{2}}{\sum_{i=1}^{n}|y_{i}|^{n+\sigma_{i}}} dy \leq C_{1}\left(n \right) ,
$$
if $C=C\left( n\right)>0$ and $r_{k}:= C 2^{-k}$, we can argue as in Corollary 9.2 in \cite{CS}.  
\end{proof}
    
\begin{corollary}
\label{Coroll Barr function 2}
Given $r > 0 $, $R > 1$ and $\sigma_{0} \in \left( 0,  2\right)$, there exist $s>0$ and $p >0$ such that the function
$$
g\left( x \right) = \min \left( s^{-p}, \  | T_{r}^{-1} |^{-p}\right)  
$$
satisfies
$$
M^{-}g\left( x \right) \geq 0 
$$
for $\sigma_{0} < \sigma_{\min}$ and $ x \in E_{r,R}\setminus E_{r, 1} $, where $p = p\left( n, \lambda, \Lambda, R\right)$ and $s = s\left( n, \lambda, \Lambda, \sigma_{0}, R \right)$.
\end{corollary}
\begin{proof}
Considering the anisotropic scaling 
$$
g\left( x \right) = f\left( T_{r}^{-1}x\right), \quad x \in \mathbb{R}^{n}, 
$$
we have $T_{r}^{-1}\left( E_{r,R}\setminus E_{r, 1}\right) = B_{R}\setminus B_{r}$. Furthermore, changing variables, we estimate 
$$M^{-}g\left(  x \right)  =  r^{-1} |\det T_{r}| M^{-}f\left( T_{r}^{-1}x \right)  \geq  0, $$
for all $x \in E_{r,R}\setminus E_{r, 1}$.
\end{proof}
    
\begin{lemma}
\label{Exist barr func 5}
Given $\sigma_{0}\in \left(0, 2 \right)$, there is a function $\Psi:\mathbb{R}^{n} \rightarrow \mathbb{R}$ satisfying
\begin{enumerate}
\item $\Psi$ is continuous in $\mathbb{R}^{n}$;
\medskip
\item $\Psi = 0$ for $x \in \mathbb{R}^{n} \setminus E_{\frac{1}{4},3\sqrt{n}}$;
\medskip
\item $\Psi> 3$ for $x \in \mathcal{R}_{\frac{1}{4},3}$; 
\medskip
\item $M^{-}\Psi\left( x\right) > - \phi\left( x \right) $ for some positive function $\phi \in C_{0}\left( E_{\frac{1}{4},1} \right)$ for $\sigma_{0} < \sigma_{\min}$.
\end{enumerate}
\end{lemma}
\begin{proof}
We define the function $\Psi:\mathbb{R}^{n} \rightarrow \mathbb{R}$ by
$$
\Psi\left( x \right) = \tilde{c} \left \{ 
\begin{array}{lll}
0, & \text{ in } & \mathbb{R}^{n} \setminus E_{\frac{1}{4},3\sqrt{n}} \\
| T^{-1}_{\frac{1}{4}}x |^{-p} -  \left( 3\sqrt{n}\right)^{-p}  & \text{ in } & E_{\frac{1}{4},3\sqrt{n}} \setminus E_{\frac{1}{4},1} \\
q_{p,\sigma},  & \text{ in } & E_{\frac{1}{4},1},
\end{array}
\right.
$$
where $q_{p,\sigma}$ is a quadratic function with different coefficients in different directions so that $\Psi$ is $C^{1,1}$ across $E_{\frac{1}{4},1}$.  Choose $ \tilde{c}>0$ such that $\Psi > 3 $ in $\mathcal{R}_{\frac{1}{4},3}$. By Lemma \ref{I is C^{1,1}}, 
$$M^{-}\Psi \in C\left( E_{\frac{1}{4},3\sqrt{n}} \right)$$ 
and, from Corollary \ref{Coroll Barr function 2}, we get $M^{-}\Psi \geq 0$ in $\mathbb{R}^{n} \setminus E_{\frac{1}{4},1}$. The lemma is proved.
\end{proof}
    
\section{Harnack inequality and regularity} \label{Harnack Inequality Section}

The next lemma is the fundamental tool towards the proof of the Harnack inequality. It bridges the gap between a pointwise estimate and an estimate in measure.

\begin{lemma} 
\label{Point Estimates 1} 
Let $0 < \sigma_{0} < 2$. If $\sigma_{\min} \in \left( \sigma_{0}, 2 \right) $, then there exist constants $\varepsilon_{0} > 0$, $0 < \varsigma < 1 $, and $M > 1$, depending only $\sigma_{0}$, $\lambda$, $\Lambda$ and $n$, such that if
\begin{enumerate}
\item $u \geq 0$ in $\mathbb{R}^{n}$;
\medskip
\item $u\left( 0 \right) \leq 1$;
\medskip
\item $M^{-}u \leq \varepsilon_{0}$ in $E_{\frac{\left( 3\sqrt{n}\right)^{n+2}}{4},1}$,
\end{enumerate} 
then 
$$| \left\lbrace u \leq M \right\rbrace \cap Q_{1}| > \varsigma.$$
\end{lemma}

\begin{proof}

Let $v= \Psi - u$ and let $\Gamma$ be the concave envelope of $v$ in $E_{\frac{\left( 3\sqrt{n}\right)^{n+2}}{4},3}$. We have 
$$
M^{+}v \geq  M^{-}\Psi -  M^{-}u \geq - \phi - \varepsilon_{0} \quad \text{in} \ \ E_{\frac{\left( 3\sqrt{n}\right)^{n+2}}{4},1}. 
$$
Applying Theorem \ref{ABP Nonlocal theorem} to $v$ (anisotropically scaled), we obtain a family of rectangles $\mathcal{R}_{j}$ such that 
$$
 \sup \limits_{E_{\frac{\left( 3\sqrt{n}\right)^{n+2}}{4},1}}v \leq C \left |\nabla \Gamma \left(E_{\frac{\left( 3\sqrt{n}\right)^{n+2}}{4},1}\right) \right |^{\frac{1}{n}}. 
$$
Thus, by Theorem \ref{ABP Nonlocal theorem} and condition (3) in Lemma \ref{Exist barr func 5}, we obtain
\begin{eqnarray*}
\sup \limits_{E_{\frac{\left( 3\sqrt{n}\right)^{n+2}}{4},1}}v &\leq & C \left |\nabla \Gamma \left(  E_{\frac{\left( 3\sqrt{n}\right)^{n+2}}{4},1} \right) \right |^{\frac{1}{n}} \\ 
&\leq & C_{1} \left( \sum_{i=1}^{n} \left | \nabla \Gamma \left( \mathcal{R}_{j} \right) \right | \right)^{\frac{1}{n}} \\ 
&\leq &  C_{1} \left( \sum_{i=1}^{n} \left( \max_{\overline{\mathcal{R}}_{j}} \left( \phi + \varepsilon_{0}\right) ^{+} \right)^{n}\left | \mathcal{R}_{j} \right |\right)^{\frac{1}{n}}\\ 
&\leq &  C_{1}\varepsilon_{0} + \left( \sum_{i=1}^{n} \left( \max_{\overline{\mathcal{R}}_{j}} \left( \phi \right) ^{+} \right)^{n}\left | \mathcal{R}_{j} \right |\right)^{\frac{1}{n}}.
\end{eqnarray*}
Furthermore, since $\Psi > 3$ in $E_{\frac{\left( 3\sqrt{n}\right)^{n+2}}{4},1} \supset \mathcal{R}_{\frac{1}{4},3}$ and $u\left( 0 \right) \leq 1$, we get
$$
2 \leq  C_{1}\varepsilon_{0} + \left( \sum_{i=1}^{n} \left( \max_{\overline{\mathcal{R}}_{j}} \left( \phi \right)^{+} \right)^{n}\left | \mathcal{R}_{j} \right |\right)^{\frac{1}{n}}.
$$
If $\varepsilon_{0} >0$ is small enough, we have
\begin{equation}
\label{PE 1}
c \leq  \left( \sum \limits_{\mathcal{R}_{j} \cap E_{\frac{1}{4},1}\neq \emptyset}  \left |\mathcal{R}_{j} \right |\right),
\end{equation}
where we used that $\phi$ is supported in $ E_{\frac{1}{4},1}$. We also have that the diameter of $\mathcal{R}_{j}$ is bounded by $\left( \rho_{0}=\frac{1}{C}\right)^{\frac{1}{n+2}}$. Then, if $\mathcal{R}_{j} \cap E_{\frac{1}{4},1}\neq \emptyset$ we have $C\tilde{\mathcal{R}}_{j} \subset B_{\frac{1}{2}}$. By Theorem \ref{ABP Nonlocal theorem}, we get 
\begin{eqnarray}
\label{PE 2}
& & \left | \left\lbrace y \in C\tilde{\mathcal{R}}_{j}: v\left( y \right) \geq \Gamma \left( y \right) - C\rho_{0}^{\frac{2}{n+2}} \right\rbrace \right | \nonumber \\
& \geq &\left | \left\lbrace y \in C\tilde{\mathcal{R}}_{j} : v\left( y \right) \geq \Gamma \left( y \right) - Cd^{2}_{j}  \right\rbrace \right | \nonumber\\ 
& \geq & \varsigma \left | \mathcal{R}_{j} \right |,
\end{eqnarray}
where we used that $Cd^{2}_{j} < C\rho_{0}^{\frac{2}{n+2}}$. For each rectangles $\mathcal{R}_{j}$ that intersects $E_{\frac{1}{4},1}$ we consider $C\tilde{\mathcal{R}_{j}}$. The family $\left\lbrace C\tilde{\mathcal{R}_{j}}\right\rbrace$ is an  open covering for $\bigcup_{i=1}^{m}\overline{\mathcal{R}}_{j}$. We consider a subcover with finite overlapping (Lemma \ref{covering lemma}) that also covers $\bigcup_{i=1}^{m}\overline{\mathcal{R}}_{j}$. Then, using \eqref{PE 1} and \eqref{PE 2} we obtain
\begin{eqnarray*}
\label{PE 3}
& & \left | \left\lbrace y \in B_{\frac{1}{2}} : v\left( y \right) \geq \Gamma \left( y \right) - C\rho_{0}^{\frac{2}{n+2}} \right\rbrace \right | \\
& \geq & \left | \bigcup_{j=1}^{m} \left\lbrace y \in  C\tilde{\mathcal{R}}_{j} : v\left( y \right) \geq \Gamma \left( y \right) - C\rho_{0}^{\frac{2}{n+2}} \right\rbrace \right | \\ 
& \geq & C_{1}\sum_{j=1}^{m}\left | \left\lbrace y \in  C\tilde{\mathcal{R}}_{j} : v\left( y \right) \geq \Gamma \left( y \right) - C\rho_{0}^{\frac{2}{n+2}} \right\rbrace \right |\\ 
& \geq & C_{1}c_{1}.
\end{eqnarray*}
We recall that $B_{\frac{1}{2}} \subset Q_{1}$ and $\Gamma \geq 0$. Hence, if $M:= \sup \limits_{B_{\frac{1}{2}}}\Psi + C\rho_{0}^{\frac{2}{n+2}}$, we have
\begin{eqnarray*}
\label{PE 4}
\left | \left\lbrace y \in Q_{1} : u\left( y \right) \leq M \right\rbrace \right | & \geq & \left | \left\lbrace y \in B_{\frac{1}{2}} : u\left( y \right) \leq M  \right\rbrace \right | \\ 
& \geq &\left | \left\lbrace y \in B_{\frac{1}{2}} : v\left( y \right) \geq \Gamma \left( y \right) - C\rho_{0}^{\frac{2}{n+2}} \right\rbrace \right | \\ 
& \geq & c.
\end{eqnarray*}
\end{proof}

The next lemma is crucial to iterate Lemma \ref{Point Estimates 1} and to obtain the $L_{\varepsilon}$ decay of the distribution function $\lambda_{u} := \left |\left\lbrace  u > t \right\rbrace \cap B_{1} \right |$. Since our scaling is anisotropic, the following Calder\'on-Zygmund decomposition is performed with boxes that satisfy the covering lemma of Caffarelli-Calder\'on (Lemma \ref{covering lemma}). We can then apply Lebesgue's differentiation theorem having these boxes as a differentiation basis.

If $Q$ is a dyadic cube different from $Q_{1}$, we say that $Q_{pred}$ is the predecessor of $Q$ if $Q$ is one of the $2^{n}$ cubes obtained from dividing $Q_{pred}$. We recall from section \ref{ABP Estimate} that if $Q$ is a cube then $\tilde{Q}$ is the cube corresponding to $Q$.

\begin{lemma} [Calder\'on-Zygmund]
\label{ANISOTROPIC DECOMPOSITION CALDERÓN-ZYGMUND} 
Let $A \subset B \subset Q_{1}$ be measurable sets and $0 < \delta < 1$ be such that
\begin{enumerate}
\item $\left | A \right | \leq \delta$;
\medskip
\item if $Q$ is a dyadic cube such that $\left | A \cap \tilde{Q} \right | > \delta \left | \tilde{Q} \right |$, then $\left( \tilde{Q}\right)_{pred} \subset B$.
\end{enumerate} 
Then 
$$\left | A \right |  \leq \delta C \left | B \right |,$$
where $C>0$ is a constant depending only on $n$.  
\end{lemma}

\begin{proof}
Just as in \cite[Lemma 4.2.]{CC}, using Lebesgue's differentiation theorem, we obtain a sequence of boxes $R_{j}$ satisfying
\begin{enumerate}
\item $\left | A \cap R_{j} \right | \leq  \delta \left | R_{j} \right |$;
\medskip
\item $A \subset \bigcup_{j=1}^{\infty} R_{j}$.
\end{enumerate}
Then, we have
$$\left | A  \right | \leq \sum_{j=1}^{\infty} \left | A \cap R_{j} \right | \leq \delta \sum_{j=1}^{\infty} \left | R_{j} \right | \leq C\delta \left | B  \right |,$$
where $C=C\left( n \right)>0$ is the constant from Lemma \ref{covering lemma}. 
\end{proof}

\begin{lemma} 
\label{Point Estimates 2}
Let $u$ be as in Lemma \ref{Point Estimates 1}. Then
$$
\left | \left\lbrace u > M^{k} \right\rbrace \cap Q_{1} \right | \leq C\left( 1 - \varsigma \right)^{k}, \quad k=1, \dots,
$$
where $M$ and $\varsigma$ are as in Lemma \ref{Point Estimates 1}. Thus, there exist positive universal constants $d$ and $\varepsilon$ such that
$$
| \left\lbrace u \geq t \right\rbrace \cap Q_{1}| \leq dt^{-\varepsilon}, \quad \forall t > 0.
$$
\end{lemma}

Using standard covering arguments we get the following theorem.
\begin{theorem}
\label{Point Estimates 3}
Let $u \geq 0$ in $\mathbb{R}^{n}$, $u\left( 0 \right) \leq 1$ and $M^{-}u \leq \varepsilon_{0}$ in $B_{2}$. Suppose that $\sigma_{\min} \geq \sigma_{0}$ for some $\sigma_{0} >0$. Then
$$
| \left\lbrace u \geq t \right\rbrace \cap B_{1}| \leq Ct^{-\varepsilon}, \quad \forall t > 0,
$$
where $C=C\left( n, \lambda, \Lambda, \sigma_{0} \right) >0$ and $\varepsilon=\varepsilon\left( n, \lambda, \Lambda, \sigma_{0} \right) >0$.
\end{theorem}

\begin{remark}
For each $s > 0$, we will denote $E^{j}_{r,s}:= E_{r^{n+\sigma_{j}},s}$. Let $u \geq 0$ in $\mathbb{R}^{n}$ and $M^{-}u \leq C_{0}$ in $E^{j}_{r,2}$, with $0< r \leq 1$. We consider the anisotropic scaling 
$$
v\left( x \right) = \dfrac{u\left( T_{j,r}x\right) }{u\left( 0 \right) + C_{0}r^{\left[ \left( n-1 \right) - \sum_{i=1}^{n-1}\frac{n + \sigma_{n}}{n + \sigma_{i}}\right]}r^{\sigma_{j}} }, \quad x \in \mathbb{R}^{n}, 
$$
where $T_{j,r}:\mathbb{R}^{n}\rightarrow \mathbb{R}^{n}$ is defined by 
$$
T_{j,r}e_{i}:= \left \{ 
\begin{array}{lll}
r e_{j}  , & \text{ for } & i=j \\
r^{\frac{n+\sigma_{j}}{n+\sigma_{i}}}e_{i}, & \text{ for } & i \neq j. 
\end{array}
\right.
$$
We have $v \geq 0$ in $\mathbb{R}^{n}$, $v\left( 0 \right) \leq 1$ and $T_{j,r}\left( B_{2}\right) \subset E^{j}_{r,2}$. Moreover, changing variables, we estimate 
$$M^{-}v\left(  x \right)  =   \dfrac{r^{\sigma_{j}} r^{\left[ \left( n-1 \right) - \sum_{i=1}^{n-1}\frac{n + \sigma_{j}}{n + \sigma_{i}}\right]}}{u\left( 0 \right) + C_{0}r^{\left[ \left( n-1 \right) - \sum_{i=1}^{n-1}\frac{n + \sigma_{j}}{n + \sigma_{i}}\right]}r^{\sigma_{j}}} M^{-}u\left( T_{j,r}x \right)  \leq  1,$$
for all $x \in B_{2}$.
\end{remark}
Then, using the anisotropic scaling $T_{j, r}$ and Theorem \ref{Point Estimates 3} we have the following scaled version.
\begin{theorem} [Pointwise Estimate]
\label{Point Estimates}
Let $u \geq 0$ in $\mathbb{R}^{n}$ and $M^{-}u \leq C_{0}$ in $E^{j}_{r,2}$. Suppose that $\sigma_{\min} \geq \sigma_{0}$ for some $\sigma_{0} >0$. Then
$$
| \left\lbrace u \geq t \right\rbrace \cap E^{j}_{r,1}| \leq \nonumber C|E^{j}_{r,1}|\left( u\left( 0 \right) + C_{0}r^{\left[ \left( n-1 \right) - \sum_{i=1}^{n-1}\frac{n + \sigma_{j}}{n + \sigma_{i}}\right]}r^{\sigma_{j}} \right)^{\varepsilon} t^{-\varepsilon}  \quad \forall t > 0
$$
where $C=C\left( n, \lambda, \Lambda, \sigma_{0} \right) >0$ and $\varepsilon=\varepsilon\left( n, \lambda, \Lambda, \sigma_{0} \right) >0$.
\end{theorem}

We are now ready to prove the Harnack inequality.

\begin{theorem}[Harnack Inequality]
\label{Harnack Inequality}
Let $u \geq 0$ in $\mathbb{R}^{n}$, $M^{-}u \leq C_{0}$, and $M^{+}u \geq -C_{0}$ in $B_{2}$. Suppose that $\sigma_{\min} \geq \sigma_{0}$, for some $\sigma_{0} >0$. Then 
$$u \leq C \left( u\left( 0 \right) + C_{0}  \right) \quad  in \ \ B_{\frac{1}{2}}.$$
\end{theorem}
\begin{proof}
Without loss of generality, we can suppose that $u\left( 0 \right) \leq 1$ and $C_{0} = 1$. Let 
$$\tau = \frac{n\left( n + \sigma_{\max} \right)}{\varepsilon \left(  n + \sigma_{\min} \right)}, $$ 
where $\varepsilon > 0$ is as in Theorem \ref{Point Estimates 3}. For each $\vartheta >0$, we define the function
$$
f_{\vartheta}\left( x \right) := \vartheta \left( 1 - | x | \right)^{-\tau}, \quad x \in B_{1}. 
$$
Let $t>0$ be such that $u \leq f_{t}$ in $B_{1}$. There is an $x_{0} \in B_{1}$ such that $u\left( x_{0} \right) = f_{t}\left( x_{0} \right)$. Let $d := \left( 1 - | x_{0} | \right)$ be the distance from $x_{0}$ to $\partial B_{1}$. 
  
If $\sigma_{\max}=\sigma_{i_{0}}$ and $E^{\max}_{r, s}\left( x_{0}\right):=E^{i_{0}}_{r, s}\left( x_{0}\right)$, for all $s >0$, we will estimate the portion of the ellipsoid $E^{\max}_{r, 1}\left( x_{0}\right)$ covered by $\left\lbrace u > \frac{u\left( x_{0} \right) }{2}\right\rbrace $ and by $\left\lbrace u < \frac{u\left( x_{0} \right) }{2}\right\rbrace $. As in \cite{CS}, we will prove that $t>0$ cannot be too large. Thus, since $\tau \leq \dfrac{2n}{\varepsilon}$, we conclude the proof of the theorem. By Theorem \ref{Point Estimates 3}, we have
$$
\left | \left\lbrace u > \frac{u\left( x_{0} \right) }{2}\right\rbrace \cap B_{1} \right | \leq C  \left | \dfrac{2}{u\left( x_{0} \right) }\right |^{\varepsilon} = C t^{-\varepsilon}d^{n} \leq C_{1}  t^{-\varepsilon} \left( r^{\frac{n+\sigma_{\max}}{n+\sigma_{\min}}}\right)^{n}  , 
$$
where $r=\frac{d}{2}$. Thus, we get  
\begin{equation}
\label{Harnack Inequality Estimate 1}
\left | \left\lbrace u > \frac{u\left( x_{0} \right) }{2}\right\rbrace \cap E^{\max}_{r, 1}\left(  x_{0} \right)  \right | \leq  C_{1}  t^{-\varepsilon}|E^{\max}_{r, 1}|. 
\end{equation} 
Now we will estimate $\left | \left\lbrace u > \frac{u\left( x_{0} \right) }{2}\right\rbrace \cap E^{\max}_{r, 1}\left( x_{0} \right)  \right |$, where $0 < \theta < 1$. Since 
$$
| x | \leq  | x - x_{0} | + | x_{0} |,  \quad \forall x \in \mathbb{R}^{n},
$$
we have
$$\left( 1 - | x | \right) \geq \left[ d -  \frac{d\theta}{2} \right],$$
for $x \in B_{r \theta }\left( x_{0} \right)$. Hence, if $x \in B_{r \theta }\left( x_{0} \right)$, we get
$$u\left( x \right) \leq \nonumber f_{t}\left( x\right) \leq t \left( 1 - | x | \right)^{-\tau} \leq u\left( x_{0} \right) \left(1 - \frac{\theta}{2} \right)^{-\tau}. $$
Then, since $M^{+}u \geq - 1$, the function
$$v\left( x \right) =  \left( 1 - \frac{\theta}{2} \right)^{-\tau}u\left( x_{0} \right) - u\left( x \right) $$ 
satisfies
$$v \geq 0 \quad \text{in} \ B_{r \theta }\left( x_{0} \right)  \quad \text{and} \quad M^{-}v \leq 1.$$
We will consider the function $w:=v^{+}$. For $x \in \mathbb{R}^{n}$ we have
$$M^{-}w\left( x \right) = M^{-}v\left( x \right) + \left( M^{-}w\left( x \right) - M^{-}v\left( x \right) \right)$$
and
\begin{eqnarray*}
\dfrac{ M^{-}w\left( x \right) - M^{-}v\left( x \right)}{c_{\sigma}} & = &  \lambda \int_{\mathbb{R}^{n}} \dfrac{\delta^{+}\left( w, x, y\right) - \delta^{+}\left( v, x, y\right) }{\sum_{i=1}^{n}|y_{i}|^{n+\sigma_{i}}}dy \\ 
&  & + \Lambda \int_{\mathbb{R}^{n}} \dfrac{\delta^{-}\left( v, x, y\right) - \delta^{-}\left( w, x, y\right) }{\sum_{i=1}^{n}|y_{i}|^{n+\sigma_{i}}}dy \\ 
& = & I_{1} + I_{2},
\end{eqnarray*}
where $I_{1}$ and $I_{2}$ represent the two terms in the right-hand side above. Using the elementary equality
$$v^{+}\left( x + y\right) = v\left( x+y\right) + v^{-}\left( x +y \right),$$
and denoting $\delta_{w}:=\delta\left( w, x, y\right)$ and $\delta_{v}:=\delta\left( v, x, y\right)$, we obtain
$$\delta^{+}_{w} = \delta_{v} + v^{-}\left( x - y \right) + v^{-}\left( x + y \right).$$
Thus, taking in account that 
$$\delta_{w}^{+} \geq \delta_{v}^{+} \quad \text{and} \quad \delta_{v} = \delta_{v}^{+} - \delta_{v}^{-},$$
we estimate
\begin{eqnarray}
\label{Harnack Inequality Estimate 6}
I_{1} & = & - \lambda \int_{\left\lbrace \delta_{w}^{+} > \delta_{v}^{+}\right\rbrace } \dfrac{\delta_{v}^{-}}{\sum_{i=1}^{n}|y_{i}|^{n+\sigma_{i}}}dy  \nonumber\\
& & +  \lambda \int_{\left\lbrace \delta_{w}^{+} > \delta_{v}^{+} \right\rbrace } \dfrac{v^{-}\left( x + y\right) + v^{-}\left( x - y\right) }{\sum_{i=1}^{n}|y_{i}|^{n+\sigma_{i}}}dy
\nonumber  \\
& \leq & \Lambda \int_{\left\lbrace \delta_{w}^{+} > 0 \right\rbrace } \dfrac{v^{-}\left( x + y\right) + v^{-}\left( x - y\right) }{\sum_{i=1}^{n}|y_{i}|^{n+\sigma_{i}}}dy.
\end{eqnarray}
Analogously, we get
\begin{eqnarray}
\label{Harnack Inequality Estimate 7}
I_{2} & = &  \Lambda \int_{\left\lbrace \delta_{v}^{-} > 0\right\rbrace \cap \left\lbrace \delta_{w}^{-} \neq \delta^{-}_{v} \right\rbrace } \dfrac{\delta_{v}^{-} - \delta_{w}^{-}}{\sum_{i=1}^{n}|y_{i}|^{n+\sigma_{i}}}dy \nonumber\\  
&  & + \Lambda \int_{\left\lbrace \delta_{v}^{-} = 0 \right\rbrace \cap \left\lbrace \delta_{w}^{-} \neq \delta^{-}_{v} \right\rbrace} \dfrac{v^{-}\left( x + y\right) + v^{-}\left( x - y\right) }{\sum_{i=1}^{n}|y_{i}|^{n+\sigma_{i}}}dy \nonumber \\ 
& \leq & \Lambda \int_{\left\lbrace \delta_{v}^{-} > 0\right\rbrace \cap \left\lbrace \delta_{w}^{-} \neq \delta^{-}_{v} \right\rbrace } \dfrac{- \delta_{v} - \delta_{v}^{-}}{\sum_{i=1}^{n}|y_{i}|^{n+\sigma_{i}}}dy.
\end{eqnarray}
We also have 
\begin{eqnarray}
\label{Harnack Inequality Estimate 8}
- \delta_{v}^{-} - \delta_{w}^{-} & = & 2v\left( x \right) - \left( v\left( x + y \right) + v\left( x - y \right)\right) - \delta_{w}^{-} \nonumber\\ 
& = & 2v\left( x \right) - \left[ \left( v^{+}\left( x + y \right) + v^{+}\left( x - y \right)\right) \right. \nonumber \\
& & \left. - \left( v^{-}\left( x + y \right) + v^{-}\left( x - y \right)\right) \right] \nonumber\\ 
& = & \left( -\delta_{w} - \delta_{w}^{-} \right) +  v^{-}\left( x + y \right) + v^{-}\left( x - y \right) \nonumber\\ 
& = & -\delta^{+}_{w} + v^{-}\left( x + y \right) + v^{-}\left( x - y \right).
\end{eqnarray}
Then, from \eqref{Harnack Inequality Estimate 8} and \eqref{Harnack Inequality Estimate 7}, we obtain
\begin{eqnarray}
\label{Harnack Inequality Estimate 9}
I_{2} & \leq &  - \Lambda \int_{\left\lbrace \delta_{v}^{-} > 0\right\rbrace \cap \left\lbrace \delta_{w}^{-} \neq \delta^{-}_{v} \right\rbrace } \dfrac{ \delta_{w}^{+}}{\sum_{i=1}^{n}|y_{i}|^{n+\sigma_{i}}} dy \nonumber \\  
&  & + \Lambda \int_{\left\lbrace \delta_{v}^{-} > 0\right\rbrace \cap \left\lbrace \delta_{w}^{-} \neq \delta^{-}_{v} \right\rbrace} \dfrac{v^{-}\left( x + y\right) + v^{-}\left( x - y\right) }{\sum_{i=1}^{n}|y_{i}|^{n+\sigma_{i}}}dy \nonumber\\ 
& \leq & \Lambda \int_{\left\lbrace \delta_{w}^{-} \geq 0 \right\rbrace} \dfrac{v^{-}\left( x + y\right) + v^{-}\left( x - y\right) }{\sum_{i=1}^{n}|y_{i}|^{n+\sigma_{i}}}dy.
\end{eqnarray}
Hence, using \eqref{Harnack Inequality Estimate 6}, \eqref{Harnack Inequality Estimate 9}, and changing variables, we find
\begin{eqnarray*}
\dfrac{ M^{-}w\left( x \right) - M^{-}v\left( x \right)}{c_{\sigma}} & \leq &  \Lambda \int_{\mathbb{R}^{n}} \dfrac{v^{-}\left( x + y\right) + v^{-}\left( x - y\right) }{\sum_{i=1}^{n}|y_{i}|^{n+\sigma_{i}}}dy \\ 
& = & -2 \Lambda \int_{\left\lbrace v\left( x + y \right) < 0 \right\rbrace } \dfrac{v\left( x + y\right)}{\sum_{i=1}^{n}|y_{i}|^{n+\sigma_{i}}}dy.
\end{eqnarray*}
Moreover, if $x \in B_{\frac{r\theta}{2}}\left( x_{0} \right) $, we have
$$\dfrac{M^{-}w\left( x \right) - M^{-}v\left( x \right)}{c_{\sigma}}  \leq   2 \Lambda \int_{\mathbb{R}^{n} \setminus B_{r \theta }\left( x_{0} - x \right)} \dfrac{-v\left( x + y\right)}{\sum_{i=1}^{n}|y_{i}|^{n+\sigma_{i}}}dy$$ 
$$ \leq  2 \Lambda \int_{\mathbb{R}^{n} \setminus B_{r \theta }\left( x_{0} - x \right)} \dfrac{\left( u\left( x + y\right) - \left( 1 - \frac{\theta}{2} \right)^{-\tau}u\left( x_{0} \right)\right)^{+}}{\sum_{i=1}^{n}|y_{i}|^{n+\sigma_{i}}}dy.$$
If $\iota > 0$ is the largest value such that $u\left( x \right) \geq \iota \left( 1 - \vert 4 x \vert^{2} \right)$, then there is a point $x_{1} \in B_{\frac{1}{4}}$ such that $u\left( x_{1} \right) = \left( 1 - \vert 4 x_{1} \vert^{2} \right)$. Moreover, since $u\left( 0 \right)\leq 1$, we get $\iota \leq 1$. Then, we have 
$$c_{\sigma} \int_{\mathbb{R}^{n}} \dfrac{\delta^{-} \left( u, x_{1}, y\right)}{\sum_{i=1}^{n}|y_{i}|^{n+\sigma_{i}}}dy \leq c_{\sigma} \int_{\mathbb{R}^{n}} \dfrac{\delta^{-} \left( \left( 1 - \vert 4 x \vert^{2} \right), x_{1}, y\right)}{\sum_{i=1}^{n}|y_{i}|^{n+\sigma_{i}}}dy \leq C,$$
where the constant $C>0$ is independent of $\sigma_{i}$. Moreover, since $M^{-}u\left( x_{1}\right) \leq 1 $, we find
$$c_{\sigma} \int_{\mathbb{R}^{n}} \dfrac{\delta^{+} \left( u, x_{1}, y\right)}{\sum_{i=1}^{n}|y_{i}|^{n+\sigma_{i}}}dy \leq  C.$$
Recall that $u\left( x_{1} - y \right)\geq 0$ and $u\left( x_{1} \right)\leq 1$. Thus, we obtain
$$c_{\sigma} \int_{\mathbb{R}^{n}} \dfrac{\left(  u\left( x_{1} + y \right) - 2 \right)^{+} }{\sum_{i=1}^{n}|y_{i}|^{n+\sigma_{i}}}dy \leq  C.$$
Since $t>0$ is large enough, we can suppose that $ u\left( x_{0} \right) > 2$. Let 
$$x \in E^{\max}_{  \frac{r\theta}{2},1}\left( x_{0} \right)\subset B_{\frac{r \theta}{2} }\left( x_{0} \right)$$ 
and 
$$y \in \mathbb{R}^{n} \setminus B_{r \theta }\left( x_{0} - x \right) \subset \mathbb{R}^{n} \setminus E^{\max}_{ \frac{r\theta}{2},1} \left( x_{0} - x \right).$$
Then, we have the inequalities 
\begin{eqnarray*}
& & \sum_{i=1}^{n}|\left( y + x + x_{1} \right)_{i}|^{n+\sigma_{i}} \\
& \leq & C \left(  \sum_{i=1}^{n}|y_{i}|^{n+\sigma_{i}} + \sum_{i=1}^{n}|x_{i}|^{n+\sigma_{i}} + \sum_{i=1}^{n}|\left( x_{1} \right)_{i}|^{n+\sigma_{i}} \right)\\
& \leq & C \sum_{i=1}^{n}|y_{i}|^{n+\sigma_{i}} + 2C
\end{eqnarray*}
and 
\begin{eqnarray*}
| y_{i}| & \geq & | \left( y - \left( x_{0} - x \right)\right)_{i} | - | \left( x_{0} - x \right)_{i} |\\ 
& \geq & \frac{\left( r\theta\right)^{\frac{n + \sigma_{\max}}{n + \sigma_{i}}}}{2}. 
\end{eqnarray*}    
Then, taking into account the obvious equalities
$$u\left( x + y\right) - \left( 1 - \frac{\theta}{2} \right)^{-\tau}u\left( x_{0} \right) =  u\left( x + x_{1} + y - x_{1} \right) - \left( 1 - \frac{\theta}{2} \right)^{-\tau}u\left( x_{0} \right),$$
and
$$ \dfrac{1}{\sum_{i=1}^{n}|y_{i}|^{n+\sigma_{i}}} = \left( \sum_{i=1}^{n}|\left( y + x + x_{1} \right)_{i}|^{n+\sigma_{i}}\right)^{-1} \dfrac{\sum_{i=1}^{n}|\left( y + x + x_{1} \right)_{i}|^{n+\sigma_{i}}}{\sum_{i=1}^{n}|y_{i}|^{n+\sigma_{i}}}, $$
we estimate
$$2 \Lambda \int_{\mathbb{R}^{n} \setminus B_{r \theta }\left( x_{0} - x \right)} \dfrac{\left( u\left( x + y\right) - \left( 1 - \frac{\theta}{2} \right)^{-\tau}u\left( x_{0} \right)\right)^{+}}{\sum_{i=1}^{n}|y_{i}|^{n+\sigma_{i}}}dy  \leq  C_{1}\left( \theta r \right) ^{-\left( n + \sigma_{\max} \right) }.$$
Thus, we have
$$M^{-}w \leq C_{1}\left( \theta r \right) ^{-\left( n + \sigma_{\max} \right) } \quad \text{in} \ E^{\max}_{  \frac{r\theta}{2},1}\left( x_{0} \right).$$
Applying Theorem \ref{Point Estimates} to $w$ in $E^{\max}_{  \frac{r\theta}{2},1}\left( x_{0} \right) \subset B_{\frac{r \theta }{2}}\left( x_{0} - x \right)$ and using that 
$$w\left( x_{0}\right) =  \left( \left( 1 - \frac{\theta}{2} \right)^{-\tau} - 1 \right) u\left( x_{0} \right),$$ 
we get
\begin{eqnarray}
\label{Harnack Inequality Estimate 16}
& & \left | \left\lbrace u >  \dfrac{u\left( x_{0} \right)}{2} \right\rbrace \cap E^{\max}_{  \frac{r\theta}{2},\frac{1}{2}} \right | \nonumber \\
& =  &   \left | \left\lbrace w > \left[ \left( 1 - \frac{\theta}{2} \right)^{-\tau} - \dfrac{1}{2} \right] u\left( x_{0} \right) \right\rbrace \cap E^{\max}_{  \frac{r\theta}{2},\frac{1}{2}} \right | \nonumber\\ 
& \leq &  C \left | E^{\max}_{  \frac{r\theta}{2},\frac{1}{2}} \right |  \left[ \left( \left( 1 - \frac{\theta}{2} \right)^{-\tau} - \dfrac{1}{2} \right) u\left( x_{0} \right) + C_{1} \left( r\theta \right)^{- n -C_{2}} \right]^{\varepsilon} \nonumber\\ 
& & \cdot  \left[ \left( \left( 1 - \frac{\theta}{2} \right)^{-\tau} - \dfrac{1}{2} \right) u\left( x_{0} \right)\right]^{-\varepsilon}\nonumber \\ 
& \leq & \nonumber C \left | E^{\max}_{  \frac{r\theta}{2},\frac{1}{2}} \right |  \left[ \left( \left( 1 - \frac{\theta}{2} \right)^{-\tau} - \dfrac{1}{2} \right) u\left( x_{0} \right) + C_{1}\left( r\theta \right)^{- C\left( n \right) } \right]^{\varepsilon}\nonumber \\ 
& & \cdot  \left[ \left( \left( 1 - \frac{\theta}{2} \right)^{-\tau} - \dfrac{1}{2} \right) u\left( x_{0} \right)\right]^{-\varepsilon},
\end{eqnarray}
where 
$$C_{2}=\left[ \sum_{i=1}^{n-1}\frac{n + \sigma_{\max}}{n + \sigma_{i}} - \left( n-1 \right)\right]$$ 
and where we have used that $0 < C_{2} \leq C_{1}\left( n \right)$.
Thus, using \eqref{Harnack Inequality Estimate 16} and the elementary inequalities
$$ \left[ \left( \left( 1 - \frac{\theta}{2} \right)^{-\tau} - \dfrac{1}{2} \right) u\left( x_{0} \right) + C_{1}\left( r\theta \right)^{ -C\left( n \right)} \right]^{\varepsilon}$$ 
$$\leq \left( \left( 1 - \frac{\theta}{2} \right)^{-\tau} - \dfrac{1}{2} \right)^{\varepsilon} u\left( x_{0} \right)^{\varepsilon} +  C_{1} \left( r\theta \right)^{- C\left( n \right) \varepsilon}$$
and
$$\left( 1 - \frac{\theta}{2} \right)^{-\tau} - \frac{1}{2} \geq \left( 1 - \frac{\theta}{2} \right)^{-\frac{n}{\varepsilon}} - \frac{1}{2} \geq \frac{1}{2},$$
for $\theta > 0$ sufficiently small, and yet
$$C_{3} \theta^{-C\left( n \right)\varepsilon} r^{-C\left( n \right)\varepsilon}u\left( x_{0} \right)^{-\varepsilon} \left(  \left( 1 - \frac{\theta}{2} \right)^{-\tau} - \frac{1}{2}\right)^{-\varepsilon} $$
$$ \leq  \nonumber C_{4} \theta^{-C\left( n \right)\varepsilon} r^{-C\left( n \right)\varepsilon}u\left( x_{0} \right)^{-\varepsilon} \leq C_{5}\theta^{-C\left( n \right)\varepsilon} t^{-\varepsilon} d^{n\left[ 1 - \tilde{C} \varepsilon \right]}  \leq  C_{6}\theta^{-C\varepsilon} t^{-\varepsilon}, $$
we obtain
$$\left | \left\lbrace u >  \dfrac{u\left( x_{0} \right)}{2} \right\rbrace \cap E^{\max}_{  \frac{r\theta}{2},\frac{1}{2}} \right | \leq C \left | E^{\max}_{  \frac{r\theta}{2},\frac{1}{2}} \right| \left[ \left( \left( 1 - \frac{\theta}{2} \right)^{-\tau} - 1 \right)^{\varepsilon} + \theta^{-C\varepsilon} t^{-\varepsilon} \right]. $$
Now we choose $\theta > 0$ sufficiently small such that
\begin{eqnarray*}
C \left |E^{\max}_{  \frac{r\theta}{2},\frac{1}{2}} \right | \left[ \left( 1 - \frac{\theta}{2} \right)^{-\tau} - 1 \right]^{\varepsilon} & \leq & C \left |E^{\max}_{  \frac{r\theta}{2},\frac{1}{2}} \right | \left[ \left( 1 - \frac{\theta}{2} \right)^{-\frac{2n}{\varepsilon}} - 1 \right]^{\varepsilon}\\
&  \leq & \dfrac{1}{4} \left |E^{\max}_{  \frac{r\theta}{2},\frac{1}{2}} \right |.
\end{eqnarray*}
Having fixed $\theta > 0$ (independently of $t$), we take $t>0$ sufficiently large such that
$$C \left |E^{\max}_{\frac{r \theta }{2},\frac{1}{2}} \right | \theta^{-C \varepsilon} t^{-\varepsilon} \leq \dfrac{1}{4} \left |E^{\max}_{\frac{r \theta }{2},\frac{1}{2}} \right |.$$
Then, using \eqref{Harnack Inequality Estimate 16}, we find
$$ \left | \left\lbrace u >  \dfrac{u\left( x_{0} \right)}{2} \right\rbrace \cap E^{\max}_{  \frac{r\theta}{2},\frac{1}{2}} \right |  \leq  \dfrac{1}{4} \left |E^{\max}_{  \frac{r\theta}{2},\frac{1}{2}} \right |.$$
Hence, we have, for $t > 0$ large,
\begin{eqnarray*} 
\left | \left\lbrace u <  \dfrac{u\left( x_{0} \right)}{2} \right\rbrace \cap E^{\max}_{  \frac{r\theta}{2},\frac{1}{2}} \right |  & \geq &  c \theta^ { 1 +  \sum_{i=1}^{n-1}\frac{n + \sigma_{\max}}{n + \sigma_{i}}} \left |E^{\max}_{  r,1} \right | \\& \geq &  c_{2} \left |E^{\max}_{ r,1} \right |,
\end{eqnarray*}
which is a contradiction to \eqref{Harnack Inequality Estimate 1}.
\end{proof}

As a consequence of the Harnack inequality we obtain the $C^{\gamma}$ regularity.
  
\begin{theorem}[$C^{\gamma}$ estimates] \label{estimates}
\label{ estimates}
Let $u$ be a bounded function such that
$$ M^{-} u \leq C_{0} \quad \text{and} \quad M^{+} u \geq - C_{0}  \ \text{in}  \ B_{1}.$$
If $(0,2) \ni \sigma_{0} < \sigma_{\min}$, then there is a positive constant $0 < \gamma < 1$, that depends only $n$, $\lambda$, $\Lambda$ and $\sigma_{0}$, such that $u \in C^{\gamma}\left( B_{1/2}\right)$ and 
$$| u |_{C^{\gamma}\left( B_{1/2}\right)} \leq C \left( \sup \limits_{\mathbb{R}^{n}}| u | + C_{0}  \right),$$
for some constant $C>0$.
\end{theorem}

The next result is a consequence of the arguments used in \cite{CS} and Theorem \ref{estimates}. As in \cite{CS}, if we suppose a modulus of continuity of $K_{\alpha\beta}$ in measure, so as to make sure that faraway oscillations tend to cancel out, we obtain the interior $C^{1, \gamma}$ regularity for solutions of equation $Iu = 0$. 

\begin{theorem}[$C^{1,\gamma}$ estimates]
Suppose that $0 < \sigma_{0} < \sigma_{\min}$. There exists a constant $\tau_{0} > 0$, that depends only on $\lambda$, $\Lambda$, $n$ and $\sigma_{0}$, such that
$$ \int_{\mathbb{R}^{n}\setminus B^{\tau_{0}}}\dfrac{| K_{\alpha\beta}\left( y \right) - K_{\alpha\beta}\left( y - h \right)|}{|h|} dy \leq C_{0}, \quad \ \text{whenever} \ |h| < \frac{\tau_{0}}{2} . $$
If $u$ is a bounded function satisfying $Iu=0$ in $B_{1}$, then there is a constant $0 < \gamma < 1$, that depends only $n$, $\lambda$, $\Lambda$ and $\sigma_{0}$, such that $u \in C^{1,\gamma}\left( B_{1/2}\right)$ and 
$$| u |_{C^{1,\gamma}\left( B_{1/2}\right)} \leq C  \sup \limits_{\mathbb{R}^{n}}| u |,$$
for some constant $C=C\left( n, \lambda, \Lambda, \sigma_{0}, C_{0} \right)>0$.
\end{theorem}

\begin{remark} We can also get $C^{\gamma}$ and $C^{1, \gamma}$ estimates for truncated kernels, i.e., kernels that satisfy \eqref{Kernel cond 2} only in a neighborhood of the origin. Let $\mathcal{L}$ be the class of operators $L_{\alpha\beta}$ such that the corresponding kernels $K_{\alpha\beta}$ have the form
$$K_{\alpha\beta}\left( y \right) = K_{\alpha\beta, 1}\left( y \right) + K_{\alpha\beta, 2}\left( y \right) \geq 0,$$
 where
$$\dfrac{\lambda c_{\sigma} }{\sum_{i=1}^{n}|y_{i}|^{n+\sigma_{i}}} \leq K_{\alpha\beta, 1}\left( y \right) \leq      \dfrac{\Lambda c_{\sigma}}{\sum_{i=1}^{n}|y_{i}|^{n+\sigma_{i}}}$$
and $K_{\alpha\beta, 2} \in L^{1}\left( \mathbb{R}^{n}\right)$ with $\Vert K_{\alpha\beta, 2} \Vert_{L^{1}\left( \mathbb{R}^{n}\right)}\leq c_{0}$, for some constant $c_{0}>0$. The class $\mathcal{L}$ is larger than $\mathcal{L}_{0}$ but the extremal operators $M_{\mathcal{L}}^{-}$ and $M_{\mathcal{L}}^{+}$ are controlled by $M^{+}$ and $M^{-}$ plus the $L^{\infty}$ norm of $u$ (see Lemma 14.1 and Corollary 14.2 in \cite{CS}). Thus the interior $C^{\gamma}$ and $C^{1,\gamma}$ regularity follow.
\end{remark}

\bigskip

\noindent \textbf{Acknowledgements.} This work was done in the framework of the UT Austin$|$Portugal program CoLab. LC supported by NSF. RL supported by CNPq-Brasil. RL and JMU partially supported by the Centro de Ma\-te\-m\'a\-ti\-ca da Universidade de Coimbra (CMUC), funded by the European Regional Development Fund through the program COMPETE and by the Portuguese Government through the FCT - Funda\c c\~ao para a Ci\^encia e a Tecnologia under the project PEst-C/MAT/UI0324/2011, and FCT project UTAustin/MAT/0035/2008. JMU partially supported by FCT projects UTA-CMU/MAT/0007/2009 and PTDC/MAT-CAL/0749/2012 and FCT grant SFRH/BSAB/1273/2012.


\begin{thebibliography}{99}

\bibitem{BK} R.F. Bass and M. Kassmann, Harnack inequalities for non-local operators of variable order, \textit{Trans. Amer. Math. Soc.} \textbf{357} (2005), 837-850.

\bibitem{BK2} R.F. Bass and M. Kassmann,  H\"older continuity of harmonic functions with respect to operators
of variable order, \textit{Comm. Partial Differential Equations} \textbf{30} (2005), 1249-1259.

\bibitem{BL} R.F. Bass and D.A. Levin, Harnack inequalities for jump processes, \textit{Potential Anal.} \textbf{17} (2002), 375-388.

\bibitem{CS} L.A. Caffarelli and L.  Silvestre, Regularity theory for fully nonlinear
integro-differential equations, {\it Comm. Pure Appl. Math.} \textbf{62} (2009), 597-638.

\bibitem{CC} L.A. Caffarelli and X. Cabr\'e, \textit{Fully nonlinear elliptic equations}. American Mathematical Society Colloquium Publications, \textbf{43}. American Mathematical Society, Providence, R.I., 1995.

\bibitem{CCal} L.A. Caffarelli and C.P. Calder\'on, Weak type estimates for the Hardy-Littlewood maximal functions, \textit{Studia Mathematica} \textbf{49} (1974), 217-223.

\bibitem{CCal2} L.A. Caffarelli and C.P. Calder\'on, On Abel summability of multiple
Jacobi series, \textit{Colloq. Math.} \textbf{30} (1974), 277-288.

\bibitem{S} L. Silvestre, H\"older estimates for solutions of integro-differential equations like the fractional
Laplace, \textit{Indiana Univ. Math. J.} \textbf{55} (2006), 1155-1174.

\bibitem{R} N. Reich, Anisotropic operator symbols arising from multivariate jump processes, \textit{Integral Equations Operator Theory} \textbf{63} (2009), 127-150.

\bibitem{SO} H.M. Soner, Optimal control with state-space constraint II, \textit{SIAM J. Control Optim.} \textbf{24} (1986), 1110-1122.

\bibitem{Son} R. Song and Z. Vondracek, Harnack inequality for some classes of Markov processes, \textit{Math. Z.} \textbf{246} (2004), 177-202.

\end{thebibliography}
\end{document}